\documentclass[a4paper,12pt,twoside,fleqn]{amsart}
\usepackage{amsmath,amsthm,amsfonts}
\usepackage[T1]{fontenc}
\usepackage{graphics,graphicx,xcolor}

\usepackage{tikz}
\usetikzlibrary{calc,chains,arrows,positioning,intersections,automata,decorations.pathmorphing,backgrounds,positioning,fit,matrix,trees}
\usepackage{tikz-cd}

\usepackage{amsmath}
\usepackage{amsthm}
\usepackage{graphicx}
\usepackage{amsfonts}
\usepackage{amssymb}

\newtheorem{thm}{Theorem}[section]

\newtheorem{cor}[thm]{Corollary}

\newtheorem{lem}[thm]{Lemma}
\newtheorem{prop}[thm]{Proposition}

\newtheorem{proposition}[thm]{Proposition}

\newtheorem{definition}[thm]{Definition}

\theoremstyle{definition}

\newtheorem{ex}[thm]{Example}
\newtheorem{exs}[thm]{Examples}

\newtheorem{rem}{Remark}
\newtheorem{rems}[rem]{Remarks}

\newcommand{\Cal}{\mathcal}
\newcommand{\cal} {\mathcal}

\newcommand{\fff}{\rightarrow}

\newcommand{\R}{{\mathbb{R}}}

\newcommand{\Q}{{\mathbb{Q}}}

\newcommand{\Z}{{\mathbb{Z}}}
\newcommand{\N}{{\mathbb{N}}}
\newcommand{\T}{{\mathbb{T}}}

\newcommand{\modu} {{\rm \, mod \ } 1}

\newcommand{\dd}{\operatorname{d}}

\def\a{{\underline a}} 
 \def\k{{\underline k}}

\newcommand{\beq}{\begin{equation}}
\newcommand{\eeq}{\end{equation}}

\def \Card {{\rm Card }}
\def \diam {\rm diam \,}

\def\e{{\rm e}}

\newcommand{\GG}{\mathcal G}

\def \hb {\hfill \break}

\def\eop{\qed}
\def\Proof {\vskip -2mm {\it Proof}.}

\def\0{{\underline 0}} 

\title[Cocyles over rotations] {Ergodicity of cocyles over 2-dimensional rotations}

\date {\today}

\author{Nicolas Chevallier and Jean-Pierre Conze} 
\address{Nicolas Chevallier, \hfill \break University of Haute Alsace, Mulhouse, France} \email{nicolas.chevallier@uha.fr}
\address{Jean-Pierre Conze, \hfill \break IRMAR, CNRS UMR 6625, \hfill \break University of Rennes, Campus de Beaulieu, 35042 Rennes Cedex, France} 
\email{jean-pierre.conze@univ-rennes.fr}


\begin{document}

\maketitle

\begin{abstract} 
We study recurrence and ergodicity of cocycles with values in $\R^d$, $d \geq 1$, over rotations by badly approximable irrational numbers 
on $\T^\rho$, $\rho > 1$. The discontinuities of the functions generating the cocycles also satisfy a Diophantine condition.
For simplicity of notation we mainly consider the cases $\rho = 2$, $d=1$ and 2. 
\end{abstract}

\tableofcontents

\section*{\bf Introduction}

Let $T:X\fff X$ be an ergodic measure preserving transformation on a probability space $(X, {\cal B}, \mu)$.
Let $\varphi$ be a measurable function on $X$ with values in $\R^d$, $d \geq 1$. 

The ergodic sums of $\varphi$ under the iteration of $T$, denoted by $\varphi_n$ or $S_n \varphi$, are defined as 
$$\varphi_n(x) := \sum_{j=0}^{n-1} \varphi(T^j x), n \geq 1, \ \varphi_0(x) = 0.$$ 
The sequence $(\varphi_n, n \geq 0)$ will be called a ``cocycle'' (over the dynamical system $(X,\mu, T)$) and denoted by $(T, \varphi)$, or $(\varphi_n)$,
or $(\varphi_{n, \alpha})$, when $T=T_\alpha$ is a rotation by $\alpha$ on a torus. 

We denote by $\tilde T_\varphi$ the skew-product map (also called ``cylinder map'') 
$$\tilde T_\varphi: (x, z) \to (Tx, z + \varphi(x)),$$ 
acting on $X \times \R^d$ endowed with the infinite invariant measure $\tilde \mu$ product of $\mu$ by the Lebesgue measure $\lambda$ (also denoted $dz$) on $\R^d$.

The cocycle $(T, \varphi)$ is said to be ergodic if the dynamical system $(X \times \R^d, \tilde \mu, \tilde T_\varphi)$ is ergodic.

In what follows, after general reminders, we take for $T$ a rotation $T_\alpha$ on the torus $X = \T^\rho= \R^\rho/\Z^\rho$, $\rho \geq 1$, 
with its Haar measure (denoted $\mu$ or $dx$):
$$T = T_\alpha: x=(x_1,  ..., x_\rho) \to (x_1 +\alpha_1, ..., x_\rho + \alpha_\rho),$$
where $\alpha = (\alpha_1, ..., \alpha_\rho) \in \T^\rho$  is totally irrational (i.e., $1, \alpha_1, ..., \alpha_\rho$ are linearly independent over $\Q$).

Given a cocycle $(T, \varphi)$, the main questions are: is it recurrent, is it ergodic?

A first remark is that, under a mild Diophantine condition on $\alpha$, too much regularity for $\varphi$ is an obstruction to ergodicity. 
Consequently, the presence of discontinuities plays a role in the construction of explicit ergodic cocycles.

Nevertheless, let us mention that three examples of ergodic cocycles over a 1-dimensional rotation are given in \cite{Kr74}: 
an analytic, a ``smooth'' and  a continuous cocycle. The latter example is constructed over an arbitrary irrational rotation.

For $\rho = 1$, in particular in the class of step functions, many examples of ergodic cocycles have been given from the late Seventies and later
(cf. for instance \cite{{Or83}, {Fr00}, Co09, {CoPi14}}). For a one-dimensional rotation, Koksma's inequality gives a uniform bound along the denominators of the rotation 
for ergodic sums of centered functions with bounded variation. It provides a way to prove the existence of non-trivial essential values and then ergodicity.  

For $\rho > 1$, there are fewer results due to the lack of such an inequality (cf. \cite{Yo80}). 
An alternative approach is based on Lebesgue density theorem, existence of recurrence times of the cocycle
and control of its discontinuities which introduces ``bad Diophantine approximation'' conditions.

The main aim here is to give examples of ergodic cocycles over rotations on $\T^\rho$, $\rho > 1$. 
The method of proof will be based on Lebesgue density theorem and recurrence times as mentioned above. 
For simplicity of notation we will present mainly examples on $\T^2$. In addition, we also review and extend some results on recurrence of cocycles over rotations. 

The main recurrence result of Section \ref{recSection}, Theorem  \ref{thm:recurrence}, applies to all rotations outside a  small exceptional set, 
while the ergodicity results of Section \ref{ergo}, Theorems \ref{thm:ergodicity1}, \ref{xymod} and \ref{ergoCpt}, require badly approximable rotations and discontinuities.

\section{\bf Preliminaries}\label{prelim}

\subsection{Recurrence and essential values} \label{recEss}

\

For the sake of completeness, we start with reminders summarizing some results, in particular on recurrence.
A basic reference on cocycles is K. Schmidt's seminal work \cite{Sch77}. 

Let $(\varphi_n)$ be a cocycle generated by a measurable function $\varphi$ 
with values in $\R^d$ over an ergodic dynamical system $(X, {\cal B}, \mu, T)$.
\footnote{In this general subsection, in view of Lemma \ref{compact}, $(X, {\cal B}, \mu)$ is a Radon measure space. The values of $\varphi$ 
could be in a locally compact group $G$, but we restrict ourselves to $G = \R^d$ provided with a norm $| \  |$.} 

Recall that $(\varphi_n)$ is recurrent if, for every neighborhood $V$ of the origin in $\R^d$, for $\mu$-a.e. $x \in X$, 
there is a strictly increasing sequence $(n_k(x))$ in $\N$ such that $\varphi_{n_k(x)} (x) \in V$.
It is recurrent if and only if $\tilde{T}_{\varphi}$ is conservative.
It is transient if $\lim_n |\varphi_n(x)| = +\infty$ for $\mu$-a.e. $x$.
A cocycle $(\varphi_n)$ over an ergodic dynamical system is either recurrent or transient

Recall also that $a \in \R^d \cup \{\infty\}$ is called {\it an essential value} of $(\varphi_n)$ if, for every neighborhood $V$ of $a$, 
for every measurable subset $B$ of positive measure in $X$, there is $n \in \N$ such that
\begin{eqnarray}
\mu(B\cap T^{-n} B \cap \{x: \varphi_n(x) \in V\}\bigr) > 0. \label{essval0}
\end{eqnarray}
We denote by ${\overline {\cal E}}(\varphi)$ the set of essential values and by ${\cal E}(\varphi)$ the set of finite essential values.
Observe that ${\cal E}(\varphi)$ contains always $0$ and that ${\overline {\cal E}}(\varphi) = \{0, \infty\}$ if the cocycle is transient.

Suppose that $(\varphi_n)$ is recurrent.
The group ${\mathcal P}(\varphi)$ of periods of the measurable $\tilde T_\varphi$-invariant functions
is a closed subgroup of $\R^d$ which coincides with ${\mathcal E}(\varphi)$. (See \cite{Sch77} or \cite{Aa97}).

{\it Therefore, proving ergodicity of a cocycle $(\varphi_n)$ with values in $\R^d$ amounts to showing that $\mathcal E(\varphi)$ contains elements generating 
a dense subgroup of $\R^d$.}

\vskip 2mm 
\goodbreak
{\it Induced map and induced cocycle}

Let $B \subset X$ be a set of positive measure. On $B$ equipped with the measure $\mu_B = \mu(B)^{-1} \mu_{|B}$, 
the induced transformation is $T_B(x) = T^{R(x)}(x)$, with $R(x) = R_B(x) := \inf\{j \ge 1: T^j x \in B\}$. 
It is well defined for a.e. $x \in B$ if the system is conservative, in particular (by Poincar\' e recurrence property) if it has a finite measure. 
Clearly if $T$ is conservative, then $T_B$ is conservative for every $B$ of positive measure.
The successive return times of a point $x$ in $B$ are $ R_1(x) =R(x), R_2(x) = R(x) + R(T^{R(x)}x), ..., R_n(x) = R(x) + R_{n-1}(T^{R(x)}x), ...$.

Let $\varphi$ be a measurable function on $X$. The "induced" cocycle (for the induced map $T_B$ on $B$) is, for $n \geq 1$, 
\begin{eqnarray*}
&&\varphi_n^B(x) := \varphi^B(x) + \varphi^B(T_B x)\,\cdots \, + \varphi^B(T_B^{n-1}x) =\varphi_{ R_n(x) }(x), x \in B,\\
&&\text{ with } \varphi^B(x) := \sum_{j= 0}^{R(x)- 1} \varphi(T^jx) = \varphi_{R(x)}(x).
\end{eqnarray*} 
We see that $a \in \R^d \cup \{\infty\}$ is an essential value if, for every neighborhood $V$ of $a$, 
and every measurable subset $B$ of positive measure in $X$, there is $n \in \N$ such that $\mu(\{x \in B: \varphi_n^B(x) \in V\}) > 0$.

\begin{rems} \label{rem1} a) A cocycle $(\varphi_n)$ is recurrent, if and only if, for each neighborhood $V$ of the origin 
and each $B \subset X$ of positive measure, there exists $n \geq 1$ such that 
\begin{eqnarray}\label{recu}
\mu(B\cap T^{-n}B\cap (\varphi_n\in V))>0. \label{rec2def}
\end{eqnarray}

b) ${\overline {\cal E}}(\varphi) =\{ 0 \}$ if and only if $\varphi$ is a coboundary (cf. \cite{Sch77}), 
meaning that there exists a measurable function $\psi: X \to \R^d$ such that $\varphi  = \psi - \psi \circ T$.

c) Two cocycles which differ by a coboundary have the same set of essential values.

d) A transient cocycle is never ergodic when $(X, \mu, T)$ is aperiodic (i.e. such that the set of periodic points is $\mu$-negligible).
Indeed, let $a \in \R^d$ be different from the origin and $V$ a neighborhood of $a$. 
By transience there is a set $A$ of positive measure such that, for some $N \geq 1$, $\varphi_n(x) \not \in V$ for $n \geq N$ and all $x \in A$. 
By Rohklin's lemma for aperiodic dynamical systems, there is a set $B \subset A$  such that the return time in $B$ is $> N$. 
This shows that (\ref{essval0}) is not satisfied and $a \not \in {\mathcal E}(\varphi)$.
Hence $(\varphi_n)$ is not ergodic. 
\end{rems}

The following lemma will be useful in the proof of ergodicity. 
\begin{lem}\label{compact} \cite[Proposition 3.8]{Sch77} If $K\subset \R^d$ is a compact set such that $K\cap \cal E(\varphi)=\emptyset$, 
there exists a set $B$ of positive measure such that $\mu(B\cap T^{-n}B\cap(\varphi_n \in K))=0, \forall n\in\Z$.
\end{lem}
\proof For the sake of completeness, we give a proof. It uses induced cocyles.

The hypothesis implies the existence, for every $z \in K$, of a subset $B_z$ of positive measure in $X$ and of a neighborhood $U_z$ of the origin such that 
$\varphi_n^{B _z}(x) \not \in U_z+z, \forall n \geq 0, \forall x \in B_z$.

Let $V_{z}$ be a neighborhood of the origin such that $V_{z} +V_{z} \subset U_{z}$.
By compactness of $K$, there is a finite number of points $z_1, ..., z_r$ such that $K \subset \cup_{i=1}^r (V_{z_i} + z_i)$. 

We proceed by induction on $r$, denoting simply $B_i$, $U_i$, $V_i$ subsets and neighborhoods. 

Suppose we have constructed a subset $D=D_{r-1}$ of positive measure such that the values of the cocycle $(\varphi_{D}^n)$ never belong to $V_i + z_i$, for $i = 1, ..., r-1$.
We are going to construct a subset $D_r$ of positive measure of $D$ such that the values of $(\varphi_{D_r}^n)$ never belong to $V_r + z_r$.

Since $D_r$ is a subset of $D$, the values of the induced cocycle $(\varphi_{D_r}^n)$ are contained in those of 
$(\varphi_{D}^n)$, so they still never belong to $V_i + z_i$, for $i = 1, ..., r-1$.

The set $D_r$ is the set $B$ of the statement, since, for $x \in D_r$, we have $\varphi_{D_r}^n(x) \not \in \cup_{i=1}^r (V_i+z_i)$, 
hence $\varphi_{D_r}^n(x) \not \in K$ and we will be done.

It remains to construct $D_r$. By ergodicity of $T$, there is $k \geq 1$ and $D' \subset D$ of positive measure such that $T^k D' \subset B_r$. 
By Lusin's theorem, there is a subset $D_r$ of $D'$ of positive measure such that $\varphi_k(x) - \varphi_k(y) \in V_r, \forall x, y \in D_r$.

For $k, n \geq 0$, we have $\varphi_n(x) = \varphi_k(x) + \varphi_{n}(T^k x) - \varphi_k(T^n x)$.
Therefore, if $x, T^n x \in D_r$, then $\varphi_k(T^n x) - \varphi_k(x) \in V_r$ and therefore $ \varphi_{n}(T^k x) \in V_r+\varphi_n(x)$.

For $x \in D_r$ and $n$ such that $T^n x \in D_r$, 
as $T^k x , T^{k+n} x \in B_r$, $\varphi_{n}(T^k x)$ is a value of the induced cocycle $\varphi_{B_r}^n$ and therefore $\not \in U_r + z_r$.
It follows that $\varphi_n(x) \not \in V_r + z_r$, because, otherwise, $\varphi_{n}(T^k x) \in V_r + V_r + z_r \subset U_r+z_r$, a contradiction.

We conclude that  $\varphi_{D_r}^n$ never takes its values in $V_r + z_r$ and $D_r$ has the desired property.
\eop

\vskip 3mm
{\bf Regularity of a cocycle}

Let $(\varphi_n)$ be a recurrent cocyle. 
The function $x \to \varphi(x) \, \rm{mod} \, \cal E(\varphi)$ on $X$ defines a cocycle with values in $\R^d / \cal E(\varphi)$  
whose finite essential values, as a quotient of $\varphi$, belong to the class of $\cal E(\varphi)$, hence are trivial. 
By Remark 1.a) either $\varphi \, \rm{mod} \, \cal E(\varphi)$ is a coboundary, 
or $\R^d/\cal E(\varphi)$ is not compact and the set of essential values of  $\varphi \, \rm{mod} \, \cal E(\varphi)$ is $\{0, \infty\}$.

In the first case, the cocycle defined by $\varphi$ with values in $\R^d$ is said to be {\it regular}. There exists then 
a measurable map $\eta : X \rightarrow \R^d$ such that the cocycle $\psi:= \varphi + \eta - T\eta$ takes a.e. its values in $\cal E(\varphi)$.
Moreover, since the group of periods of the $\tilde T_\psi$-invariant functions is still $\cal E(\varphi)$, it follows that 
$\tilde T_\psi: (x, z) \rightarrow (T x, z + \psi(x))$ is ergodic for its action on $X \times\cal E(\varphi)$.
Therefore regularity for a cocycle $\varphi$ means that, if it is not ergodic, it can be reduced up to a coboundary to an ergodic cocycle with values in a closed subgroup.

\subsection{Diophantine conditions, $\rm{Bad}_{\Z}(\alpha)$} \label{prelimDioph}

\

{\bf Notation.} The Hausdorff dimension of a set $E \subset \R^d$ is denoted by $\dim_H E$.

For $u \in \R$, let $\{u\} = u - k$ if $u \in [k, k+1[, k \in \Z$, denote its fractional part 
and let $\|u\|:= \inf (\{u\}, 1 - \{u\}) = \inf_{n \in \Z} |u - n|$ denote its distance to $\Z$.
The set $\Z^2\setminus\{(0,0)\}$ is denoted by $\Z^2_*$.

For $h=(h_1, h_2)$ and $x =(x_1, x_2)$ in $\R^2$, we denote by $\langle h, x \rangle$ or $h.x$ the scalar product $h_1 x_1 + h_2 x_2$.
 
If $\alpha=\frac1{a_1+\frac1{a_2+\frac1{a_3+\ldots}}}$, with partial quotients $(a_n)_{n\geq1}$, is an irrational number,
its denominators are $q_n$: $q_0=1$, $q_1=a_1$ and $q_{n+1}=a_{n+1}q_n+q_{n-1}$ for $n\geq1$.

We recall now some facts about Diophantine properties of irrational numbers.

For $s \geq 0$, $D(s)$ denotes the set of irrational numbers $\alpha$ such that, for a finite constant $A = A(\alpha, s)$, the partial quotients of $\alpha$ satisfy
\begin{eqnarray}
a_n \leq A \, n^s, \forall n \geq 1. \label{majan}
\end{eqnarray}
By a theorem of Borel-Bernstein, a.e. $\alpha$ is in $D(s)$ for every $s> 1$.
Moreover, thanks to the inequality $q_{n+1}\|q_n\alpha\|\geq \tfrac12$, we see that for  all $n\geq 0$ and all $q_n\leq k<q_{n+1}$, we have 
$\|k\alpha\|\geq \|q_n\alpha\|\geq \tfrac1{2q_{n+1}}\geq \tfrac 1{4a_{n+1}k}.$ 
As the sequence $(q_n)$ grows at least exponentially, it follows that,  if $\alpha$ satisfies (\ref{majan}), there is a constant $c >0$ such that 
\begin{eqnarray}
\|k \alpha\| \geq {c \over k \, (\log k)^s}, \, \forall k > 1. \label{kalph1}
\end{eqnarray}

Recall also that the type (or Diophantine exponent) of an irrational number $\alpha$ is the real $\eta \ge 1$ such that
\begin{eqnarray}
\inf_k \,[k^{\eta - \varepsilon}\|k \alpha\|] =0, \ \ \inf_k [k^{\eta + \varepsilon} \|k \alpha\|] > 0, \ \forall \varepsilon >0. \label{type1}
\end{eqnarray}
The type of $\alpha$ satisfying (\ref{kalph1}) for some $s> 1$ is 1 and therefore by what precedes the type of a.e. $\alpha$ is 1.

More directly, it can be observed that, if $\alpha$ is not of type 1, there is an integer $r \geq 1$ such that $ k^{1+1/r}\|k\alpha\|\leq 1$ for infinitely many $k$.
For each $n$, the set of $\alpha$'s satisfying the latter property is negligible by the Borel-Cantelli lemma.

\vskip 3mm 
{\bf Badly approximable numbers}

Recall that a number $\theta$ is {\it badly approximable} $(\theta \in \text{Bad})$, if
\begin{eqnarray}
\exists c > 0: \|q \theta\| \geq \frac c {|q|}, \, \forall q \in \Z \setminus \{0\}. \label{bpq}
\end{eqnarray}
This is equivalent for $\theta$ to have bounded partial quotients (bpq). Clearly the type of numbers in Bad is 1. 
The set Bad has Lebesgue measure 0 and Hausdorff dimension 1, see \cite{Ja}.

The set $\rm{Bad}_{\Z}(\theta)$ of {\it badly approximable numbers with respect to an irrational $\theta$} is
\[\operatorname{Bad_{\Z}(\theta)}= \{x\in[0,1]: \exists c(x) > 0: \|q\theta-x\|\geq {c(x) \over |q|}, \forall q \in \Z \setminus \{0\}\}.\]
Observe that $\text{Bad}_{\Z}(\theta)=-\text{Bad}_{\Z}(\theta)$ and that  $0 \in \operatorname{Bad_{\Z}(\theta)}$ is equivalent to $\theta \in \text{Bad}$.

The set $\text{Bad}_{\Z}(\theta)$ has measure 0, but its Hausdorff dimension $\dim_H \text{Bad}_{\Z}(\theta)$ is 1. Actually
(cf. Proposition \ref{prop:badApprox1} in appendix), the set of $n$-tuples $(\beta_1, ..., \beta_n)$ which are in $\text{Bad}_{\Z}(\theta)$
as well as the differences $\beta_j-\beta_i$ for all $i,j, i\not = j$, is big in the sense that its Hausdorff dimension is $n$.

{\bf Role of Diophantine conditions in the question of ergodicity}

Let $(\varphi_n)$ be a recurrent cocycle with values in $\R^d$ over a rotation $x \to x + \alpha$ on $\T^\rho$.
It is easily seen that first possible obstruction to ergodicity is when some component $\varphi^i$ of the function $\varphi$ generating the cocycle is a coboundary
(meaning that there is a measurable function $\psi : \T^\rho \to \R$ such that  $\varphi^i= \psi - T_\alpha \psi$ a.e.).

The reduction of a component to a coboundary is related to the decay of its Fourier coefficients. 
For an example, when $\rho = 1$, we can use (see \cite[Lemma 2.2]{Co09} for a proof):
\begin{lem} \label{serdnm} If $\alpha$ is an irrational of type $\eta$, 
then $\displaystyle{ \sum_{k \ge 1} {1 \over k^{\eta + \delta}} {1\over \|k\alpha\|} < \infty}$ for every $\delta > 0$.
\end{lem} 
\begin{proposition} If $\alpha$ is of type $\eta$ and if $\varphi(x) = \sum_{n \not= 0} c_n(\varphi) e^{2\pi i n x}$ is such that 
$c_n(\varphi) = O(n^{-(\eta + \delta)})$ for some $\delta > 0$, then $\varphi$ is a coboundary: $\varphi = T_\alpha \psi - \psi$, with $\psi$ continuous.
\end{proposition}
\proof \ The Fourier coefficients of an integrable solution $\psi$ of the coboundary equation $\varphi = T_\alpha \psi - \psi$, are given by 
$c_n(\psi) = {c_n(\varphi) \over e^{2\pi i n \alpha} - 1}$.
By Lemma \ref{serdnm} we have
$$\sum_{n \not = 0} |c_n(\psi)|\leq \sum_{n \not= 0} { |c_n(\varphi) \over \|n\alpha\|}| \leq C \sum_{n \not= 0} {1 \over n^{\eta + \delta}} {1 \over \|n \alpha\|} < +\infty.$$
Therefore the coboundary equation has a solution which is continuous. \eop

For example, $\varphi: x \to x (1-x) - \frac 1 6$ coincides on $[0, 1]$ with the continuous, 1-periodic and 1-lipschitz function on $\R$ given by the Fourier expansion
${-1 \over\pi^2} \, \sum_{n \geq 1} {\cos (2 \pi n x) \over n^2}$.

If $\alpha$ is of type $< 2$, then this function $\varphi$ is a coboundary for the rotation by $\alpha$ and the cocycle  $(\varphi_{n, \alpha})$ in not ergodic.

{\it A non-regular BV cocycle} 
 
As an illustration of the role of Diophantine properties, let us also mention an example of a non regular (hence non ergodic) cocycle (cf. \cite{Co09}).

If $\alpha$ is an irrational $\not \in$ Bad, it can be shown that there are $\beta, r$ in $]0, 1[$ such that
$$\varphi: x \to \varphi(x) = 1_{[0,\beta]}(x \text{ mod } 1) - 1_{[0,\beta]}(x + r \text{ mod } 1)$$ 
satisfies ${\overline {\mathcal E}}(\varphi) = \{0, \infty \}$. This implies that the cocycle $(\varphi_{n, \alpha})$ is not regular 
and that $\tilde T_{\alpha, \varphi}$ is not ergodic on $\T^1 \times \Z$ endowed with the product of the Lebesgue measure on $\T^1$ by the counting measure on $\Z$.  

\section{\bf About recurrence, examples} \label{recSection}

{\bf A sufficient condition for recurrence} 

The question of recurrence for a cocycle with values in $\R^d$, $d \geq 1$ is natural and plays a key role in the proof of ergodicity. 
Let us first consider the general case of a cocycle $(\varphi_n)$ generated by a function $\varphi : X \to \R^d$, $d \geq 1$, over an ergodic dynamical systems $(X, \mu, T)$.

If $d=1$ and $\varphi$ is integrable, a necessary and sufficient condition for the recurrence of $(\varphi_n)$ is $\int_{X} \varphi \, d\mu = 0$. (cf. \cite{At76})

For $d \geq 2$, the question of recurrence is more difficult, but sometimes recurrence can be deduced from the growth rate  of the cocycle $(\varphi_n)$.
The following general lemma gives a simple sufficient condition for recurrence (cf. \cite{ChCo09}). 

For $\varphi : X \to \R^d$ in $L^2(\mu)$, denote by $\|\varphi_{n}\|_2 := (\int_X |\varphi_{n}(x)|^2 \, d\mu(x))^\frac12$ its $L^2$-norm.
\begin{lem}\label{recgen} Let $(\varphi_n)$ be a cocycle over a dynamical system $(X,{\cal B},\mu, T)$ with values in
$\R^d$. If there exist a strictly increasing  sequence of integers $(k_n)$ and a sequence of real numbers $(\delta_n > 0)$ such that:
$\lim_{n \to \infty} \mu(x: |\varphi_{k_n}(x)| \ge \delta_n) = 0\ {\rm and \ } \delta_n = o(n^{1/d})$,
then the cocycle is recurrent.

In particular, if $\varphi \in L^2(\mu)$ and $\|\varphi_{n}\|_2 = o(n^{{1\over d}})$, the cocycle is recurrent.
\end{lem} 

{\bf Cocycles over 1-dimensional rotations}

Let $\cal C$ be the class of centered real valued functions with bounded variation (BV) on $\T$.
It contains the centered step functions with a finite number of discontinuities.
If $\varphi$ belongs to $\cal C$, its Fourier coefficients  $c_{r}(\varphi)$ satisfy:
\begin{align}
c_{r}(\varphi) = {\gamma_r(\varphi) \over r}, \,\forall r \not = 0, \text{ with } K(\varphi) := \sup_{r \not = 0} |\gamma_r(\varphi)| < +\infty. \label{majC}
\end{align}
For $\varphi \in \cal C$ with variation $V(\varphi)$, for a rotation $T_\alpha$ on $\T$ and any denominator $q$ of $\alpha$, the ergodic sum $\varphi_q$ over $T_\alpha$ satisfies:
\begin{align}
\|\varphi_q\|_\infty = \sup_x |\sum_{j = 0}^{q-1}f(x+j \alpha)| \leq V(\varphi) \text{ ({\it Koksma's inequality})}. \label{DK0}
\end{align}
If $\varphi$ centered in $L^2(\T)$ satisfies (\ref{majC}), then $\|\varphi_{q}\|_2 \leq 2 \pi \, K(\varphi)$. 
Indeed, for $\psi(x) = \{x\} - \frac12$, we have $\|\psi_q\|_2 \leq \|\psi_q\|_\infty \leq V(\psi)= 1$, by (\ref{DK0}).
Hence for $\varphi$, it holds: 
\begin{eqnarray*}
\|\varphi_{q}\|_2 = (\sum_{r \not = 0} {|\gamma_{r}(\varphi)|^2 \over r^2} \, |{\sin \pi q r \alpha \over \sin \pi r \alpha}|^2)^\frac12
\leq K(\varphi) \, (\sum_{r \not = 0} {1 \over r^2} \, |{\sin \pi q r \alpha \over \sin \pi r \alpha}|^2)^\frac12 = 2 \pi \, K(\varphi) \, \|\psi_{q}\|_2 \leq 2 \pi \, K(\varphi).
\end{eqnarray*}
For $\varphi \in \cal C$, by (\ref{DK0}) we get a bound on the growth of the ergodic sum $\varphi_n$ for a.e $\alpha$:
\begin{proposition} \label{bnd2} Let $\varphi$ be a centered BV function on $\T$. If (\ref{majan}) is satisfied for some $s \geq 0$
(a condition which holds for a.e. $\alpha$), there is a constant $K_s > 0$ such that:
\begin{eqnarray}
\|\varphi_n\|_\infty \leq K_s (\log n)^{1+s}, \, \forall n \geq 2. \label{majergo0}
\end{eqnarray}
When $\alpha$ is of type 1, for every $\varepsilon > 0$ there is a constant $K(\varepsilon) > 0$ such that:
\begin{eqnarray}
\|\varphi_n\|_\infty \leq K(\varepsilon) \, n^{\varepsilon}, \, \forall n \geq 1. \label{majergo2}
\end{eqnarray}
\end{proposition}
\proof Let $(q_k)$ be the denominators of $\alpha$. For $n \geq 1$, let $m = m(n)$ be the integer such that $n \in [q_{m(n)}, \, q_{m(n)+1}[$.
As the growth of the sequence $(q_n)$ is at least exponential, we have $m(n) =O(\log n)$.

Now (\ref{majergo0}) and (\ref{majergo2}) follow easily from Koksma's inequality and the $\alpha$-Ostrowski's representation of the integers (cf \cite{Os})
which reads:
\begin{eqnarray*}
\text{ if } n <  q_{m+1}, \  n =\sum_{k=0}^{m} b_k \, q_{k}, \text{with } 0 \leq b_0 \leq a_1 -1, \ 0 \leq b_k \leq a_{k+1} \text{ for } 1 \leq k \leq m. \eop
\end{eqnarray*}
If $\varphi$ is a centered function satisfying (\ref{majC}), the previous proposition is valid with the $L^2$-norm instead of the uniform norm.

\goodbreak
\vskip 3mm
\subsection {Examples of recurrent cocycles over rotations in dimension $\geq 1$} \label{Examples}

\ 

Now we give examples where recurrence can be proved or disproved in dimension $\geq 1$. 

\vskip 3mm
\subsubsection{\bf Examples and counterexamples}

\

For cocycles over 1-dimensional rotations, (\ref{DK0}) can be used for cocycles with values in $\R^d$, $d \geq 1$. 
When the  dimension $\rho$ is  $> 1$, an estimation of the rate of growth of $\varphi_n)$ 
can be obtained in some cases by Fourier analysis methods under an hypothesis on $\alpha$. 

\vskip 2mm
{\bf Example 1.} (Case $\rho = 1$) If the cocycle $(\varphi_n)$ with values in $\R^d$ is generated over a one dimensional rotation by $\varphi$ centered with BV components 
(or more generally satisfying (\ref{majC})), then it is recurrent for any ergodic rotation and any $d \geq 1$ by Koksma's inequality.

\vskip 2mm
{\bf Example 2.} (Case $d = \rho$ > 1) Let $\varphi = (\varphi^1, ..., \varphi^d)$ be such that $\varphi^j$, for each $j$, 
is a centered BV function of the one-dimensional variable $x_j \in \T$.

In this example, the cylinder map on $\T^\rho \times \R^\rho$ : $(x, z) \to (x + \alpha, z + \varphi(x))$ is the product of the cylinder maps on $\T \times \R$:
$(x_i, z_i) \to (x_i + \alpha_i, z_i + \varphi^i(x_i))$. Recurrence follows then immediately for a large class of rotations from Lemma \ref{recgen} and Proposition \ref{bnd2}:
\begin{prop} If the components $\varphi^j$ are in the class $\cal C$  (or more generally satisfy (\ref{majC})), for $\alpha = (\alpha_1, ..., \alpha_\rho)$,
with each $\alpha_i$ of type 1, the cocycle is recurrent. 
\end{prop} 
As a.e.\,$\alpha$ is of type 1, the conclusion holds for a.e.\,$\alpha$.
However, as recalled below, for $\varphi$ with component in $\cal C$, recurrence can fail in dimension $\geq 2$ for special choices of $\alpha$.

\vskip 2mm
{\bf Example 3.} {\it Transient cocycles over a 2-dimensional rotation}

In \cite{Yo80} J.-C. Yoccoz constructed a centered transient cocycle given by an analytical function over a
particular 2-dimensional rotation. 

The following centered cocycle is another example of transient cocycle (cf. \cite[Theorem 4.1]{ChCo09}):
Let $\varphi :\T^{2}\rightarrow \R^{2}$ be the function 
$\varphi(x) = (\varphi^1(x_1), \varphi^2(x_2)), \text{ with } \varphi^1 = \varphi^2 =1_{[0,{1\over 5}]}(\{.\})-\frac{1}{5}$. 
There exists an ergodic rotation in $\T^2$, $x \rightarrow x+\alpha$, $\alpha=(\alpha_1, \alpha_2)$, such that
$$|\sum_{k=0}^{n-1} \varphi^1(x_1+k\alpha_1)| + |\sum_{k=0}^{n-1} \varphi^2(x_2+k\alpha_2)| 
\ \ {\buildrel {n \rightarrow +\infty} \over {\longrightarrow}} \ \  +\infty, \text{ for a.e. } (x_1,x_2) \in\T^2.$$
{\it A question}: Recall that a pair $\alpha\in\R^2$ is singular
if $\forall \varepsilon>0,\,\exists Q>1, \,\forall q>Q, \exists n\leq q, \,\mathrm d(n\alpha,\Z^2)\leq \varepsilon q^{-1/2}$.

Singular pairs and more generally singular vectors were defined by Khintchine who showed that the set of singular vectors is of zero Lebesgue measure. 
Recently Y. Cheung \cite{Cheu2011} showed that the set of singular pairs has Hausdorff dimension $4/3$, hence is rather small. 

In the previous examples of transient cocycles, the rotation $T_\alpha$ on $\T^2$ turns out to be associated with a singular pair $\alpha$.
Both constructions were designed to exhibit at least one $\alpha$ for which the cocycle is transient. So the fact that $\alpha$ is singular might just be a technical convenience, 
and a natural question is to construct a transient cocycle over a rotation defined by a non singular pair $\alpha \in \R^2$.

\vskip 8mm
\goodbreak
\subsection{Recurrence for a special class of functions} 

\ 

{\it Notation.} Recall the notation $|u|_+=\max(|u|,1)$, for $u \in \R$. If $\ell_1, \ell_2:\R^2\fff\R$ are two independent linear forms, 
for $h\in\R^2$ we put $R(h)=R_{\ell_1, \ell_2}(h)=|\ell_1(h)|_+|\ell_2(h)|_+$ and define for $s>1$:
\begin{eqnarray*}
W(\ell_1, \ell_2, s):=\{\alpha\in\R^2:R(h)^{s}\|h.\alpha\|\leq 1 \text{ for infinitely many }h\in\Z^2\}.
\end{eqnarray*}
\begin{definition} \label{classG} \rm We denote by ${\cal G}$ the class of centered functions $f:\T^2 \to \R$ such that there exists a finite partition of $[0,1[^2$ 
into triangles $\Delta_j$ such that $f$ has bounded continuous partial derivatives $f'_x, f'_y, f''_{xy}, f''_{yx}$ on the interior of each $\Delta_j$.
\end{definition}

\begin{thm}\label{thm:recurrence} If $\Phi=(\varphi^1,\dots,\varphi^d): \T^2 \to \R^d$ is such that each component $\varphi^i \in \GG$, then  
\hb 1) $\dim_H \{\alpha\in\R^2: (\Phi_{n, \alpha})_n \text{ is not recurrent}\}\leq 2-\frac1{2d-1}$;
\hb 2) $(\Phi_{n, \alpha})_n$ is recurrent if $\alpha$ is a totally irrational algebraic pair.
\end{thm}
{\it Remark:} The class $\cal G$ contains in particular the functions $\varphi^\Delta := 1_\Delta- \mu(\Delta)$,
where $1_\Delta$ is the indicator of a subset $\Delta$ of the 2-torus whose boundary is a finite union of segments.

In this case Theorem \ref{thm:recurrence} is related to the following result:
It is shown in \cite[Theorem 3.1]{ChCo09} that the cocycle $(\varphi_{n, \alpha}^\Delta)$ generated by $\varphi^\Delta$ over a two dimensional rotation by $\alpha$ satisfies,
for every $\gamma > 0$, for almost every $\alpha \in \T^2$ the bound $\displaystyle \|\varphi_{n, \alpha}^\Delta\|_2 = O(n^\gamma)$.

For a finite family $(\varphi^{\Delta_i}, i=1,\cdots, d)$ with sets $\Delta_i$ as $\Delta$ above, it follows from this bound and Lemma \ref{recgen}
that the $d$-dimensional cocycle $(\Phi_{n, \alpha})$ generated over the rotation by $\alpha$ on $\T^2$ 
by $\Phi = (\varphi^{\Delta_i})_{i=1,\cdots, d}$ is recurrent for a.e. $\alpha$.

Theorem \ref{thm:recurrence} improves this result. Its proof follows the same guideline. 
It will be used to show that the assumptions of Theorem \ref{xymod} below about ergodicity of some cocycles are satisfied 
by a set of rotations in $\T^2$ of Hausdorff dimension 2. 
It is based on the following two propositions whose proof is postponed to the next subsection.

\begin{proposition}\label{prop:Hausdorff} Let $\ell_1, \ell_2:\R^2\fff\R$ be two independent linear forms 
and let $R(h)=R_{\ell_1, \ell_2}(h)=|\ell_1(h)|_+|\ell_2(h)|_+$ for $h\in\R^2$. Let $1<t<2$. Then
\[\dim_H\{\alpha\in\R^2:\sum_{h\in\Z_*^2}\frac{1}{R(h)^2\|h.\alpha\|^{t}}=\infty\}\leq \frac{3+2/t}{1+2/t}.\]
\end{proposition}

\begin{proposition}\label{prop:Fourier} If $\varphi \in \cal G$, there exist a constant $C$ and $2m$ linear forms $\ell_1^\varphi,\dots, \ell_{2m}^\varphi:\R^2\fff\R$ 
such that for $k=1,\dots,m$,  $\ell_{2k-1}^\varphi, \ell_{2k}^\varphi$ are linearly independent and  the Fourier coefficients of $\varphi$ satisfy
\begin{eqnarray}
|c_n(\varphi)| \leq C\sum_{k=1}^m\frac{1}{|\ell_{2k-1}^\varphi(n)|_+|\ell_{2k}^\varphi(n)|_+}, \forall n\in\Z^2_*. \label{majCoef}
\end{eqnarray}
\end{proposition}
\proof [\bf Proof of Theorem \ref{thm:recurrence}] 

1) Let $\Phi=(\varphi^1,\dots,\varphi^d)$ be such that each $\varphi^i$ is in $\GG$ and centered. 
With $\varphi=\varphi_i$, let $\ell_1^\varphi,\dots, \ell_{2m}^\varphi$ be the linear forms given by Proposition \ref{prop:Fourier}. 

Let $0<t<2$. If $\alpha$ is not in the set
$$E^t_{\varphi}
=\left\{\alpha\in\R^2:\sum_{k=1}^m\sum_{h\in\Z^2_*}\frac{1}{(|\ell_{2k-1}^\varphi (h)|_+|\ell_{2k}^\varphi (h)|_+)^2\|h.\alpha\|^{t}}=\infty\right\},$$
then by (\ref{majCoef})
\begin{eqnarray*}
\sum_{h\in\Z^2_*}\frac{| c_h(\varphi)|^2}{\|h.\alpha\|^t}
&\leq& C^2 \sum_{h\in\Z^2_*}\big(\sum_{k=1}^m\frac{1}{|\ell_{2k-1}^\varphi (h)|_+|\ell_{2k}^\varphi (h)|_+}\big)^2\frac{1}{\|h.\alpha\|^t}\\
&\leq& C^2 \, m \sum_{k=1}^m\sum_{h\in\Z^2_*}\frac{1}{(|\ell_{2k-1}^\varphi (h)|_+|\ell_{2k}^\varphi (h)|_+)^2\|h.\alpha\|^{t}}<\infty.
\end{eqnarray*}
Since $\varphi$ is centered, we have
\begin{align*}
&\|\sum_{k=0}^{N-1}\varphi(. + k\alpha)\|^2_2 = \|\sum_{k=0}^{N-1}\sum_{h\in\Z^2_*} c_h(\varphi)e^{2i\pi \langle h, (.+k\alpha)\rangle}\|^2_2
= \sum_{h\in\Z^2_*}| c_h(\varphi)|^2|\sum_{k=0}^{N-1}e^{2i\pi k\langle h, \alpha\rangle}|^2 \\
&\leq \sum_{h\in\Z^2_*}| c_h(\varphi)|^2 \inf(N,\frac1{\|h.\alpha\|})^2
\leq \sum_{h\in\Z^2_*}| c_h(\varphi)|^2 (N^{1-t/2}\frac1{\|h.\alpha\|^{ t/2}})^2
\leq  N^{2-t}\sum_{h\in\Z^2_*}\frac{| c_h(\varphi)|^2}{\|h.\alpha\|^{t}}.
\end{align*}
It follows that, if $2-\tfrac2d<t<2$ and  $\alpha\notin \cup_{i=1}^d E^t_{\varphi_i}$, then $\|\sum_{k=0}^{N-1}\Phi(. + k\alpha)\|_2^2=O(N^{2-t})=o(N^{\frac2d})$,
which implies that the cocycle $(\Phi_{n, \alpha})_n$ is recurrent by Lemma \ref{recgen}. 

Therefore, if $\alpha \notin \bigcap_{2-\frac2d<t<2}\bigcup_{i=1}^d E^t_{\varphi_i}$, $(\Phi_n)_n$ is recurrent. 
Finally, by Proposition \ref{prop:Hausdorff},
$$\dim_H \bigl(\bigcap_{2-\frac2d<t<2} \, \bigcup_{i=1}^d E^t_{\varphi_i} \bigr)\leq \inf_{2-\frac2d<t<2}\frac{3+2/t}{1+2/t}
=\frac{3+2/(2-\frac2d)}{1+2/(2-\frac2d)}=2-\frac1{2d-1}. $$ 

2) If $\alpha$ is algebraic, by W. Schmidt's theorem \cite[Theorem 2]{WSch70} on simultaneous approximation to irrational numbers by rationals, 
$\alpha \not \in W(\ell_1, \ell_2, s)$ for $s= 1+\varepsilon$, for every $\varepsilon > 0$. By Lemma \ref{lem:Niederreiter} below it follows
$\displaystyle \sum_{h\in\Z_*^2}\frac{1}{R(h)^2\|h.\alpha\|^{t}}<\infty, \, \forall t \in ]1, \frac{2}{1+ \varepsilon}[$
and as above, taking $\varepsilon$ small enough we obtain recurrence in any dimension $d$ by Lemma \ref{recgen}. \eop

\subsubsection{Proof of Propositions \ref{prop:Hausdorff} and \ref{prop:Fourier}}

\ 

For Proposition \ref{prop:Hausdorff} we need two lemmas. 
The first one is a simple consequence of the Hausdorff-Cantelli lemma (see \cite{BeDo99}). The second lemma is adapted from Niederreiter \cite{Ni72}.
\begin{lem}\label{lem:Hausdorff} Let $\ell_1, \ell_2:\R^2\fff\R$ be two  independent linear forms and let $s>1$. 
Then, $\displaystyle \dim_H W(\ell_1, \ell_2,s)\leq\frac{3+s}{1+s}$.
\end{lem}
\proof  \ Given $h=(h_1,h_2)\in\Z_*^2$,  the set $L(h)=\{x\in\T^2: h.x \in \Z\}$  is a union of $\gcd(h_1,h_2)$  
one dimensional tori. The total length of $L(h)$ is the Euclidean norm $|h|$ of $h$. 
The set $V(s,h) := \{x\in\T^2:R(h)^s\|h.x\|\leq 1\}$ is included in a strip around $L(h)$ of width $\displaystyle \frac{1}{|h|R(h)^s}$
and can be covered with $n(h)=|h|^2R(h)^s$ balls of radius $\displaystyle r(h)=\frac{4}{|h|R(h)^s}$.

By the Hausdorff-Cantelli lemma, for $t>0$, if $\sum_{h\in\Z^2_*}n(h)r(h)^t=4^t\sum_{h\in\Z^2_*}|h|^{2-t}R(h)^{s-st}<\infty$,
then the Hausdorff dimension of $W(\ell_1, \ell_2,s)$ is $\leq t$.

For $0<t<2$, using the equivalence of norms, we obtain
\begin{align*}
\sum_{h\in\Z^2_*}|h|^{2-t}R(h)^{s-st}&\leq C\sum_{h\in\Z^2_*}(|\ell_1(h)|+|\ell_2(h)|)^{2-t}R(h)^{s-st}\\
&= C \sum_{h\in\Z^2_*}(|\ell_1(h)|+|\ell_2(h)|)^{2-t}(|\ell_1(h)|_+|\ell_2(h)|_+)^{s(1-t)}.
\end{align*}
So it suffices to bound from above the two series	$ \sum_{h\in\Z^2_*} |\ell_i(h)|^{2-t}(|\ell_1(h)|_+|\ell_2(h)|_+)^{s(1-t)}$, $i=1,2$.
The set $\Lambda=\{(\ell_1(h), \ell_2(h)):h\in\Z^2\}$ is a lattice in $\R^2$. 
 
Let $P=\{x\in\R^2:-\tfrac12\leq x_i<\tfrac12,i=1,2\}$. 
Since for all $x\in\R^2$, $\Card \, \Lambda\cap(x+P) \leq C'$ for some constant $C'$, we have for $1<t<2$,
\begin{align*}
&\sum_{h\in\Z^2_*}|\ell_1(h)|^{2-t}(|\ell_1(h)|_+|\ell_2(h)|_+)^{s(1-t)}=\sum_{(x_1,x_2)\in\Lambda\setminus\{0\}}|x_1|^{2-t}(|x_1|_+|x_2|_+)^{s(1-t)}\\
&=\sum_{(n_1,n_2)\in\Z^2} \, \sum_{(x_1,x_2)\in (\Lambda\setminus\{0\})\cap((n_1,n_2)+P)}|x_1|^{2-t}(|x_1|_+|x_2|_+)^{s(1-t)}\\
&\leq C'\sum_{(n_1,n_2)\in\Z^2}(|n_1|_+)^{2-t}(|n_1|_+|n_2|_+)^{s(1-t)}=C'\sum_{n_1\in\Z}|n_1|_+^{2-t+s(1-t)}\sum_{n_2\in\Z}|n_2|_+^{s(1-t)}.
\end{align*}
The product of the two series is finite when $s(1-t)<-1$ and $2-t+s(1-t)<-1$. Since $s > 1$, the product of the series converges 
when $2>t>\max(\frac{1+s}{s},\frac{3+s}{1+s})=\frac{3+s}{1+s}$. 

The conclusion is the same for $i=2$. \eop

\begin{lem}\label{lem:Niederreiter} Let $\ell_1, \ell_2:\R^2\fff\R$ be two  independent linear forms. Let $s>1$ 
and let $\alpha \not \in W(\ell_1, \ell_2, s)$ be totally irrational. Then,
$$\displaystyle \sum_{h\in\Z_*^2}\frac{1}{R(h)^2\|h.\alpha\|^{t}}<\infty, \, \forall t \in ]1, 2/s[.$$
\end{lem}
\proof \ Since $\alpha$ is totally irrational and since $R_{\ell_1, \ell_2}(h)^s\|h.\alpha\|\leq 1$ has only finitely many solutions $h\in\Z^2$
for $\alpha \not \in W(\ell_1, \ell_2, s)$, there exists a constant $c>0$ such that  for all $h\in\Z_*^2$, $ R(h)^s\|h.\alpha\|\geq c$.

Let us estimate the sum $\displaystyle \sum_{h\in\Z_*^2:|\ell_i(h)|_+\leq n_i,i=1,2}\frac1{\|h.\alpha\|^t}$,
for $n=(n_1,n_2)\in\Z^2$ with $n_1,n_2\geq 1$. 

Observe that $|\|x\|-\|y\||=\min(\|x-y\|,\|x+y\|)$ for any $x, y \in \R$ (where $\|u\|= \inf |u _ n|$, for $u \in \R$). 

For every pair $(h,h')$ with $h\neq h'$ and $h,h'$ both in the domain of summation, we have $|\ell_i(h\pm h')|_+\leq |\ell_i(h)|_++|\ell_i(h')|_+\leq 2n_i$.
It follows
\begin{eqnarray*}
|\|h.\alpha\|-\|h'.\alpha\||&=& \min(\|(h-h').\alpha\|,\|(h+h').\alpha\|)\\
&\geq& c\min(R(h-h')^{-s},R(h+h')^{-s}) \geq \frac{c}{(2n_1)^s(2n_2)^s}. \end{eqnarray*}		
Therefore, putting $\displaystyle \delta =\frac{c}{4^s}\frac{1}{n_1^s n_2^s}$, each interval $[k\delta,(k+1)\delta[$, $k=1,\dots,\lceil 1/\delta\rceil$ 
contains at most one point $\|h\alpha\|$ with $h$ in the domain of summation. 
We also have $\displaystyle \|h.\alpha\|\geq \frac{c}{R(h)^s}\geq\frac{c}{n_1^sn_2^s}$ for $h$ in the domain of summation.
Since $t>1$, it follows
\[\sum_{h\in\Z_*^2,|\ell_i(h)|_+\leq n_i,i=1,2}\frac1{\|h.\alpha\|^t}\leq \sum_{k\geq 1}\frac1{(k\delta)^t}\leq C\frac1{\delta^{t}}.\]
Using the inequality
\[\sum_{(n_1,n_2):n_i\geq |\ell_i(h)|_+,\, i=1,2}\frac1{n_1^{3}n_2^{3}}\geq \frac{c'}{|\ell_1(h)|_+^{2}|\ell_2(h)|_+^{2}}=\frac{c'}{R(h)^{2}},\]
satisfied for some constant $c' > 0$, we obtain
\[\sum_{h\in\Z_*^2}\frac1{R(h)^2\|h.\alpha\|^t} \leq C\sum_{h\in\Z_*^2}\frac1{\|h.\alpha\|^t}\sum_{(n_1,n_2):n_i\geq |\ell_i(h)|_+, i=1,2}\frac1{n_1^{3}n_2^{3}}.\]
Then, by permuting the order of summation, we obtain
\begin{align*}
\sum_{h\in\Z_*^2}\frac1{R(h)^2\|h.\alpha\|^t}	
&\leq {C \over c'} \sum_{(n_1,n_2)\in\N_*^2}\frac1{n_1^{3}n_2^{3}}\sum_{h\in\Z_*^2:|\ell_i(h)|_+\leq n_i,i=1,2}\frac1{\|h.\alpha\|^t}\\
&\leq C'\sum_{(n_1,n_2)\in\N_*^2}\frac1{n_1^{3}n_2^{3}} \frac1{\delta^t} \, \leq \, C'\sum_{(n_1,n_2)\in\N_*^2}\frac1{n_1^{3-ts}n_2^{3-ts}}.
\end{align*} 
The last series converges if $3-ts>1$, i.e., if $t<2/s$. \eop

\proof [\bf Proof of Proposition \ref{prop:Hausdorff}] Let $1<t < 2$ and $s\in (1,2/t)$. By Lemma \ref{lem:Hausdorff}, the Hausdorff dimension of the set $W(\ell_1, \ell_2,s)$
is $\displaystyle \leq \frac{3+s}{1+s}$. By Lemma \ref{lem:Niederreiter}, if $\displaystyle \sum_{h\in\Z_*^2}\frac{1}{R(h)^2\|h.\alpha\|^{t}}=\infty$, with $1<t < 2/s$,
then $\alpha\in W(\ell_1, \ell_2,s)$.
\[\hspace{-1.3cm}\text{Therefore, } \dim_H\{\alpha\in\R^2:\sum_{h\in\Z_*^2}\frac{1}{R(h)^2\|h.\alpha\|^{t}}=\infty\}\leq\inf_{1<s<2/t }\frac{3+s}{1+s}
=\frac{3+2/t}{1+2/t}. 
\eop\]
\proof [\bf Proof of Proposition \ref{prop:Fourier}] (Bound on the Fourier coefficients of $f$ in $\cal G$)
 
From the definition of $\cal G$, it suffices to prove (\ref{majCoef}) for $f$ supported by a triangle $\Delta$ such that $f$ has bounded continuous partial derivatives 
$f'_x, f'_y,  f''_{xx}, f''_{xy}, f''_{yx}$ on the interior of $\Delta$. Using translations, vertical axis symmetries and by cutting the triangle along a vertical line 
through one of its vertices, we can reduce to the triangles $\Delta=\Delta(a,b,c)$ where $(0,0), (a,b), (0,c)$ are the vertices of $\Delta$ and  $a,b,c$ 
are real numbers such that $0 < a,c \leq 1$ and $c-1 \leq b \leq 1$. So we are reduced to the proposition:
\begin{proposition} Let $f$ be a function supported in $\Delta(a,b, c)$ with bounded continuous partial derivatives $f'_x, f'_y, f''_{xx}, f''_{xy}, f''_{yx}$
on the interior of $\Delta(a,b,c)$.
Then its Fourier coefficients $c_f(s,t) = \int_0^1 \int_0^1 f(x,y) e^{-2\pi i (sx+ty)} \, dx dy$ satisfy,
for a finite constant $K$, 
\begin{eqnarray}
|c_f(s,t)| &\leq& K\ \bigl({1 \over |t|_+ |s|_+} + {1 \over |t|_+ |bt+ as|_+} + {1\over |t|_+ |(b-c)t+ as|_+} \bigr), \, \forall s, t \in \Z. \label{coeff0}
\end{eqnarray}
\end{proposition}
\proof  We have
$c_f(s,t) = \int_0^a I_t(x) \, e^{-2\pi i sx} dx, \text{ with } I_t(x) = \int_{\frac{b}{a} x}^{\frac{b-c}{a}x+c} f(x,y) \, e^{-2\pi i ty} \, dy.$

Fort $t \not = 0$, using integration by parts we get $\displaystyle I_t(x) = {1\over - 2\pi it} \, [A_t(x) - B_t(x) - C_t(x)]$, with
\begin{eqnarray*}
&A_t(x) = f(x,\frac{b-c}{a}x+c) \, e^{-2\pi i t(\frac{b-c}{a}x+c)}, \\
&B_t(x) =  f(x, \frac{b}{a} x) \, e^{-2\pi i t{\frac{b}{a} x}}, 
\ C_t(x) = \int_{\frac{b}{a} x}^{\frac{b-c}{a}x+c} f_y'(x, y) \, e^{-2\pi i ty} \, dy.
\end{eqnarray*}
\begin{eqnarray*}
&&\text{If } t(b-c)+ sa \not = 0, \text{ then } \\
&&\int_0^a A_t(x) e^{-2\pi i sx} \, dx = e^{-2\pi i t c} \int_0^a f(x,\frac{b-c}{a}x+c) \, e^{-2\pi i (t(\frac{b-c}{a}) + s)x} \, dx \\
&=&{a \over - 2\pi i(t(b-c)+ sa)} \, [ f(a,b) \, e^{-2\pi i (tb + sa)} -  f(0,c) \, e^{-2\pi i t c} \\
&&-  e^{-2\pi i t c} \int_0^a \bigl(f'_x(x,\frac{b-c}{a}x+c) + \frac{b-c}{a} f'_y(x,\frac{b-c}{a}x+c\bigr) \, e^{-2\pi i (t(\frac{b-c}{a}) + s)x} \, dx].
\end{eqnarray*}
\begin{eqnarray*}
&&\text{If }tb + sa \not = 0,  \text{ then } \int_0^a B_t(x) e^{-2\pi i sx} \, dx = \int_0^a f(x,\frac{b}{a}x) \, e^{-2\pi i (t \frac{b}{a} + s)x} \, dx \\
&&\hspace{-1cm} ={a \over - 2\pi i(tb + sa)} \, [ f(a,b) e^{-2i\pi(tb+sa)}-  f(0,0) 
- \int_0^a \bigl(f'_x(x,\frac{b}{a}x) + \frac{b}{a} f'_y(x,\frac{b}{a}x)\bigr) \, e^{-2\pi i (t \frac{b}{a} + s)x} \, dx].
\end{eqnarray*}
\begin{eqnarray*}
&&\text{If } s \not = 0,  \text{ then } \int_0^a C_t(x) e^{-2\pi i sx} \, dx 
= \int_0^a (\int_{\frac{b}{a} x}^{\frac{b-c}{a}x+c} f_y'(x, y) \, e^{-2\pi i ty} \, dy) \, e^{-2\pi i sx} \, dx \\
&&= {1\over - 2\pi i s} \, [- \int_0^c f_y'(0, y) \, e^{-2\pi i ty} \, dy
- \int_0^a \ {d\over dx} (\int_{\frac{b}{a} x}^{\frac{b-c}{a}x+c} f_y'(x, y) \, e^{-2\pi i ty} \, dy) \, \e^{-2\pi i sx} \, dx].
\end{eqnarray*}
The last integrand above is uniformly bounded with respect to $t, s$, as shown by
\begin{align*}
&{d\over dx} (\int_{\frac{b}{a} x}^{\frac{b-c}{a}x+c} f_y'(x, y) \, e^{-2\pi i ty} \, dy) 
= \frac{b-c}{a} f_y'(x, \frac{b-c}{a}x+c) \, e^{-2\pi i t(\frac{b-c}{a}x+c)} \\
& \ \ - \frac{b}{a} f_y'(x, \frac{b}{a} x) \, e^{-2\pi i t \frac{b}{a} x}
+ \int_{\frac{b}{a} x}^{\frac{b-c}{a}x+c} f_{yx}''(x, y) \, e^{-2\pi i ty} \, dy.
\end{align*}

The previous computation shows that, for a finite constant $K$, if $|t|\geq 1$, $ |t(b-c)+ sa| \geq 1$, $|tb+ sa| \geq 1$ and  $|s|\geq 1 $, then 
\begin{eqnarray*}
|c_f(s,t)| &\leq& {K\over |t|_+ |t(b-c)+ sa|_+} + {K\over |t|_+ |tb+ sa|_+} + {K\over |t|_+ |s|_+}.
\end{eqnarray*}
If $t\not=0$, since the integrals
$|\int_0^a A_t(x) e^{-2\pi i sx} \, dx|$, $|\int_0^a B_t(x) e^{-2\pi i sx} \, dx|$, $|\int_0^a C_t(x) e^{-2\pi i sx} \, dx|$ are bounded by some constant depending only on $f$, 
the above inequality holds even when  $ |t(b-c)+ sa| \leq 1$ or $|tb+ sa| \leq 1$ or  $|s|\leq 1 $.

Likewise, if $t = 0$ and $s \not = 0$, then, for a constant $K$, $\displaystyle |c_f(s,0)| \leq {K \over |s|_+}$. \eop

\subsection {The triangle $\Delta_0 = \{(x, y) \in [0, 1]^2: x < y \}$}

\

Since $\{x\} = x + 1$ for $x \in ]-1, 0[$, the characteristic function of $\Delta_0 = \Delta(1,1,1)$ reads 
\begin{eqnarray}
1_{\Delta_0}(x, y) = \{x - y\} + \{y\} - \{x\}, \, (x, y) \in  [0, 1]^2. \label{tri0}
\end{eqnarray}
For this special triangle a bound for the ergodic sums generated by $\varphi := 1_{\Delta_0} - \frac12$ can be obtained 
as in the proof of Theorem \ref{thm:recurrence} or by a simple method based on (\ref{tri0}): 
\begin{proposition} \label{triang} 
1) If the partial quotients of $\alpha_1$, $\alpha_2$ and $\alpha_1 - \alpha_2$ satisfy (\ref{majan})  for some $s> 0$
(a condition which holds for a.e. $\alpha = (\alpha_1, \alpha_2)$), there is a constant $C_s$ such that
\begin{eqnarray}
\|\sum_{k=0}^{n-1} \varphi(. + k\alpha_1, . + k\alpha_2)\|_\infty \leq C_s (\log n)^{1+s}, \, \forall n \geq 2. \label{Delta01}
\end{eqnarray}
2) If $\alpha_1$, $\alpha_2$ are algebraic, then for every $\varepsilon > 0$, there is a constant $C(\varepsilon)$ such that
\begin{eqnarray}
\|\sum_{k=0}^{n-1} \varphi(. + k\alpha_1, . + k\alpha_2)\|_\infty \leq C(\varepsilon) \, n^{\varepsilon}, \, \forall n \geq 2. \label{Delta02}
\end{eqnarray}
\end{proposition}
\proof 1) Putting $\psi(x) = \{x\} - \frac12$, the ergodic sums of $1_{\Delta_0} -\frac12$ are
$$\sum_{k=0}^{n-1} \psi(x-y + k(\alpha_1 - \alpha_2)) + \sum_{k=0}^{n-1} \psi(y + k \alpha_2) - \sum_{k=0}^{n-1} \psi(x + k \alpha_1).$$
If $\alpha_1$, $\alpha_2$ and $\alpha_1 - \alpha_2$ are in the set $D_s$ of irrational numbers satisfying (\ref{majan}) for some $s \geq 0$, 
then by Proposition \ref{bnd2} there is a constant $C> 0$ such that, $\forall n \geq 2$,
\begin{eqnarray*}
&&\|\sum_{k=0}^{n-1} \psi(. + k(\alpha_1 - \alpha_2))\|_\infty \leq C (\log n)^{1+s}, \, \|\sum_{k=0}^{n-1} \psi(. + k\alpha_i)\|_\infty \leq C (\log n)^{1+s}, i=1,2. 
\end{eqnarray*}
By (\ref{tri0}) of $1_{\Delta_0}$, we obtain the same bound for $\varphi$: there is $C_1$ such that (\ref{Delta01}) is satisfied.

The set $D_s$ has full measure. 
The set $D_{2,s}$ of pairs $(\alpha_1, \alpha_2)$ such that $\alpha_1$, $\alpha_2$ and $\alpha_1 - \alpha_2$ are in $D_s$ is a set of full measure in $\R^2$.
This is because $D_{2,s} = (D_s \times D_s) \cap \{(\alpha_1, \alpha_2): \alpha_1 \in D_s+\alpha_2\}$ and by Fubini the second set in the intersection has full measure.

2) If $\alpha_1, \alpha_2$ are algebraic, since $\alpha_1 - \alpha_2$ is also algebraic, by Roth's theorem, $\alpha_1, \alpha_2, \alpha_1 - \alpha_2$ are of type 1.
By Proposition \ref{bnd2}, for every $\varepsilon > 0$ there is a constant $C(\varepsilon) > 0$ such that (\ref{Delta02}) is satisfied. \eop

The triangle $\Delta_0$ will be considered again in Section \ref{Delta0}.

\section{\bf Examples of ergodic cocycles over rotations on $\T^2$} \label{ergo}

There are relatively few known examples of ergodic cocycles over a 2-dimensional rotation. Let us mention some of them:
\hb - In \cite{La94} it is shown that for $\varphi:(x,y)\in\T^2\fff \{x\}\sin2\pi y\in \R$, for uncountably many $\alpha_1$, there are uncountably many $\alpha_2$
 such that the cocycle $(\varphi_n)$ over the rotation by $(\alpha_1, \alpha_2)$  is ergodic.
\hb - Let $T_\alpha$, $\alpha =(\alpha_1, \alpha_2)$, be an ergodic rotation on $\T^2$ with $\alpha_1, \alpha_2 \in \rm{Bad}$. Let $\varphi$ be a function on $\T^2$
of the form $\varphi(x,y) = (\varphi_1(x), \varphi_2(y))$ with $\varphi_i:\T\to\Z$, $i=1,2$, centered step functions with rational discontinuities.
In \cite{CoFr11} it is shown that the $\Z^2$-cocycle $(\varphi_n)$ over $T_\alpha$ is ergodic if the jumps of $(\varphi_1, 0)$ and $(0, \varphi_2)$ generate $\Z^2$.
\hb  - In \cite{AbWu25} the ergodicity of some cocycles over rotations on $\T^2$ is shown for a class of examples quite different from those we consider here. 
The results are for rotations by $\alpha=(\alpha_1,\alpha_2)$ of Liouville type rather than badly approximable and for cocycles generated by indicators of some rectangles.

In this section, under Diophantine conditions, we prove ergodicity for two families of 1-dimensional and 2-dimensional cocycles generated 
over some rotations $T_\alpha$ on $\T^2$ by functions with (locally) non zero derivatives.

{\it Once for all, we suppose $\alpha= (\alpha_1, \alpha_2)$ totally irrational (i.e., $1, \alpha_1, \alpha_2$ linearly independent over $\Q$), 
a necessary and sufficient condition for the ergodicity of the rotation $T_\alpha$ on $\T^2$.}

\subsection{Class ${\cal F}_1$} \label{sec:F}

\ 

First we define on $\T^2$ a class ${\cal F}_1$ of $\R$-valued functions with discontinuities, but with local regularity. Then, after preliminary results,
we prove a result of ergodicity (Theorem \ref{thm:ergodicity1}).
\begin{definition} {\rm ${\cal F}_1$ is the class of centered functions $\varphi$ on $\T^2$ such that, for two finite sets depending on $\varphi$:
$J_i= \{\beta^i_0=0 \leq \beta^i_1 < ... < \beta^i_{r_i-1} \leq  \beta^i_{r_i}= 1\}$, with $r_i \geq 1$, $i=1,2$, the partial derivative
$\frac{\partial \varphi}{\partial x_1}$ exists on  the open rectangles
$\displaystyle P_{j, j'} = ]\beta^1_j, \beta^1_{j+1}[ \, \times \, ]\beta^2_{j'}, \beta^2_{j'+1}[, \, j=0,..., r_1-1, j'=0, ..., r_2-1$.
Moreover, we suppose that $\varphi$ and $\frac{\partial \varphi}{\partial x_1}$ are continuous on each $P_{j, j'}$ and can be continuously extended to its closure.}
\end{definition}
$\cal P_{\varphi}$ will denote the partition of the unit square
\footnote{In what follows, we will call ``partition of the unit square'' any finite collection of disjoint subsets of the unit square which covers it up to a Lebesgue negligible set.}
 into the rectangles $P_{j,j'}$.

Observe that, if $\varphi \in {\cal F}_1$, the limits \footnote
{It is understood that $\varphi(0_-, x_2):= \lim_{x_1 \to 1, \, x_1 < 1} \varphi(x_1,x_2)$ and $\varphi(1_+, x_2):= \lim_{x_1 \to 0, \, x_1 >0} \varphi(x_1,x_2)$.} 
$$\varphi(\beta_-, x_2):= \lim_{x_1 \to \beta, \, x_1 < \beta} \varphi(x_1,x_2), \ \varphi(\beta_+, x_2):= \lim_{x_1 \to \beta, \, x_1 > \beta} \varphi(x_1,x_2)$$ 
exist and are finite for every $\beta \in J_1$ and every $x_2 \not \in J_2$.

For every $n \geq 1$, if we write each set of numbers $(\{\beta^i_j - \ell \alpha_i\}, j \in J_i, \, 0 \leq \ell < n)$, $i=1,2$, as an ordered set of distinct points 
$(\gamma_{n,\ell}^i)_{\ell= 1,\ldots, p_{i,n}}$, with $p_{i,n} = r_i n$, then the atoms of 
$${\cal P}_{\varphi}^{n} := \cal P_{\varphi} \wedge T_\alpha^{-1}{\cal P_{\varphi}} \wedge ... \wedge T_\alpha^{-(n-1)}{\cal P_{\varphi}}$$
are the rectangles 
\begin{eqnarray}
R_{\ell, \ell'}^n = ]\gamma_{n,\ell}^1, \gamma_{n,\ell+1}^1[ \, \times \, ]\gamma_{n,\ell'}^2, \gamma_{n,\ell'+1}^2[. \label{Rll}
\end{eqnarray}
On each atom of ${\cal P}_{\varphi}^{n}$, $\varphi_n$ is continuous, the partial derivative $\frac{\partial \varphi_n}{\partial x_1}$ exists and can be extended to its closure.
\begin{exs} {\rm a) Let $\{0 \leq \beta^i_1 < ... < \beta^i_{r_1-1} \leq 1\}$, $i=1,2$, be two finite sequences in $[0,1]$, $v_j$ continuous functions, 
$\gamma_{j,j'}$ coefficients. Then the sum
$\varphi(x_1,x_2) = \sum_{j,j'} \gamma_{j,j'} [\{x_1- \beta^1_j\} \, v_{j'}(\{x_2 - \beta^2_{j'}\}) - \frac12 \int_0^1 v_j dx_2]$ is in $\cal F$. 

b) Let us taking a finite partition of the unit square into open rectangles $P_{j, j'}$ and a family $\varphi_{j, j'}$ 
such that each $\varphi_{j, j'}$ is defined and  $C^1$ on an open set containing the closure of $P_{j, j'}$. The function $\varphi$ defined by 
$\varphi \, |P_{j, j'} = \varphi_{j, j'} \, |P_{j, j'}$ (and arbitrarily on the negligible complement $[0,1] \times [0,1] \setminus \cup_{j,j'} P_{j, j'}$) is then in ${\cal F}_1$.
} \end{exs}
 
{\bf Hypothesis $H_1$ on $\alpha$ and the discontinuities $\beta^1_j $}:
$$\beta^1_j - \beta^1_{j'} \in \operatorname{Bad}_{\Z}(\alpha_1), \, \forall j, j' \in \{1, \dots, r_1\}.$$

\begin{rem} \label{rkH}
a) In particular, $0\in \operatorname{Bad}_{\Z}(\alpha_1)$, meaning that $\alpha_1 \in \rm{Bad}$. 
Recall that once for all $\alpha= (\alpha_1, \alpha_2)$ is assumed to be totally irrational.

b) There is no condition on $\beta_j^2$. It will be shown in Subsection \ref{Badly} that
\hb $\bullet$ given $\alpha_1 \in \rm{Bad}$, the set of $(\beta_j^1, \, j= 1, ..., r_1)$ satisfying condition $H_1$ has Hausdorff dimension $r_1$ in $\R^{r_1}$;
\hb $\bullet$ given $\beta^1_j$, $j=1,\dots,r_1$, the set of $\alpha$ such that $H_1$ holds has Hausdorff dimension 2.
\end{rem}

\begin{lem}\label{lem:bad} Under hypothesis $H_1$, there exist two constants $0<c\leq c'$ such that
\begin{equation}\label{defcs0}
{\frac{c}{n}} \leq \gamma_{n, \ell+1}^1- \gamma_{n,\ell}^1 \leq \frac{c'}{ n}, \forall n\geq 1, \ \ell= 1,\ldots, p_{1,n}.
\end{equation}
\end{lem}
\begin{proof} By $H_1$, there is a constant $c>0$ such that for all $j,j'\in\{1,\dots,r_1\}$ and all $k\in\{1,\dots,n\}$,	
$\|k\alpha_1-(\beta_j^1-\beta_{j'}^1)\|\geq \frac{c}{n}$.
Therefore, for all $1\leq j,j'\leq r_1$ and all $0\leq k,k'\leq n$, if $\beta_j+k\alpha_1\neq \beta_{j'}+k'\alpha_1 \mod \Z$, 
then $\|\beta_j+k\alpha_1-(\beta_{j'}+k'\alpha_1)\|\geq \min(\frac{c}{n},\delta)\geq \frac{\min(c,\delta)}{n}$,
with $\delta=\min_{j\neq j'}\|\beta_j-\beta_{j'}\|$.

For the right hand side of (\ref{defcs0}), it is enough to show that, for all positive integers $q$, 
the largest gap in $\T^1\setminus\{0,\dots,q\alpha_1\}$ is at most $\tfrac{c'}{q}$ for some constant $c'$. 
First with $j=j'$, we obtain that $\|n\alpha_1\|\geq \frac{c}{n}$, for all $n\geq 1$. 
Next, let $(\tfrac{p_n}{q_n})_{n\geq 0}$ be the sequence of convergents of $\alpha_1$. 
Since for all $n\in\N$, $q_{n+1}|q_n\alpha_1-p_n|=q_{n+1}\|q_n\alpha_1\|\leq 1$ and $q_n\|q_n\alpha_1\|\geq c$, we have 
$\tfrac{q_{n+1}}{q_n}\leq \tfrac1c$ and $|k\tfrac{p_n}{q_n}-k\alpha_1|\leq \tfrac1{q_{n+1}}$ for all $0\leq k\leq q_n$. 
This implies that the largest gap in $\T^1\setminus\{0,\dots,q_n\alpha_1\}$ is at most $\tfrac1{q_n}+\tfrac2{q_{n+1}}\leq \tfrac3{q_n}$. 
Hence, for all $q_n\leq q \leq q_{n+1}$, the largest gap in $\T^1\setminus\{0,\dots,q\alpha_1\}$ is at most $\tfrac3{q_n}\leq \tfrac{3}{cq}$.
\end{proof}

{\bf Variation of the ergodic sums of $\varphi \in {\cal F}_1$.}

We associate with $\varphi \in {\cal F}_1$ the following quantity  
$\displaystyle \lambda_1(\varphi) := \int_{\T^2} {\partial \varphi \over \partial x_1} \, dx$,
which reads
\begin{eqnarray*}
&&\sum_{j, j'} \int_{P_{j,j'}} {\partial \varphi \over \partial x_1}  \, dx
= \sum_{j=0,...,r_1-1, j'=0,...,r_2-1} \int_{\beta^2_{j'}}^{\beta^2_{j'+1}} \, [\varphi((\beta^1_{j+1})_-, x_2) - \varphi ((\beta^1_j )_+, x_2)] \, dx_2\\
&&= \sum_{j=0,...,r_1-1} \int_0^1 \, [\varphi((\beta^1_{j+1})_-, x_2) - \varphi ((\beta^1_j )_+, x_2)] \, dx_2.
\end{eqnarray*}
For instance, if $\varphi^1(x_1,x_2) = \{x_1\} \, \{x_2 \} - \frac14$, we have $r_1=r_2=1$, $\varphi^1(1_-, x_2) =\{x_2\} -\frac 14$,
$\varphi^1(0_+, x_2) = -\frac 14$, and 
$\lambda_1(\varphi^1) =  \int_0^1 [\varphi^1(1_-, x_2) - \varphi^1(0_+, x_2)] \, dx_2 = \frac 12$.

\begin{lem}\label{lem:var} Let $\varphi$ be in $\cal F_1$.  Suppose that $\lambda_1(\varphi)\neq 0$. Then, there exists $N\in\N$ such that for all integers $n\geq N$,
\hb$\bullet$ if $x$ is not in the boundary of a rectangle of the partition $\mathcal P_n$, then $\frac{\partial \varphi_n(x)}{\partial x_1}$ and $\lambda_1(\varphi)$ 
have the same sign, 
\hb$\bullet$ if $(x_1, x_2)$ and $(x_1+u_1, x_2)$ belong to interior of the same element of ${\cal P_n}$, then 
\begin{eqnarray*}
 \tfrac{1}{2}n|\lambda_1(\varphi) u_1|\leq |\varphi_n(x_1+u_1,x_2) -\varphi_n(x_1,x_2)| \leq 2n|\lambda_1(\varphi) u_1|. \label{dev}
\end{eqnarray*}
\end{lem}
\proof \ We can assume $\lambda_1(\varphi)>0$ w.l.g. 
Since the rotation $T_{\alpha}$ is uniquely ergodic and since $\frac{\partial \varphi}{\partial x_1}$ is Riemann integrable, 
the sequence of ergodic sums $(\displaystyle \frac{1}{n}\sum_{k=0}^{n-1}\frac{\partial\varphi}{\partial x_1}\circ T_{\alpha}^k, n \geq 1)$ converges uniformly 
to $\displaystyle \int_{\T^2}\frac{\partial\varphi}{\partial x_1}\,dx =\lambda_1(\varphi)>0$.
It follows that there exists an integer $N$ such that for every $n\geq N$ and every $x\in \T^2$ not in the boundary of a rectangle in $\mathcal P_n$, 
\begin{eqnarray}
\label{derivee}&&\frac12\lambda_1(\varphi)\leq\frac{1}{n}\sum_{k=0}^{n-1}\frac{\partial\varphi}{\partial x_1}( T_{\alpha}^k(x))\leq 2\lambda_1(\varphi).
\end{eqnarray}
It follows that  $\displaystyle \frac{\partial \varphi_n(x)}{\partial x_1}$ and $\lambda_1(\varphi)$ have the same sign.
  
To prove the second item, we can assume $u_1>0$. 
By the hypothesis on $(x_1,x_2)$ and $(x_1+u_1, x_2)$, for each $0\leq k<n$, their images $T_\alpha^k (x_1,x_2), T_\alpha^k (x_1+u_1, x_2)$, 
 belong to the interior of the same rectangle $P_{j, j'}$ for some $j,j'$. By definition of the class $\mathcal F$, the derivative $\frac{\partial\varphi}{\partial x_1}$ 
 exists on each segment $[T_{\alpha}^k(x_1,x_2),T_{\alpha}^k(x_1+u_1,x_2)]$, hence 
\begin{eqnarray*}
\varphi_n(x_1+u_1,x_2)=\varphi_n(x_1,x_2)+\int_{x_1}^{x_1+u_1}\sum_{k=0}^{n-1}\frac{\partial\varphi}{\partial x_1}(T_{\alpha}^k(x_1+t,x_2))\,dt.
\end{eqnarray*}
Now the second item of the lemma follows, since by (\ref{derivee}),
\begin{eqnarray*}
\frac12u_1\lambda_1(\varphi)\leq\int_{x_1}^{x_1+u_1}\frac1n\sum_{k=0}^{n-1}\frac{\partial\varphi}{\partial x_1}(T_{\alpha}^k(x_1+t,x_2))\,dt 
\leq 2u_1\lambda_1(\varphi). \eop
\end{eqnarray*}
 
\subsection{Ergodicity of $(\varphi_n)$ for $\varphi \in {\cal F}_1$}

\

\begin{thm}\label{thm:ergodicity1} Let $\varphi$ be a function in ${\cal F}_1$ such that $\lambda_1(\varphi)\not = 0$.
Let $\alpha = (\alpha_1, \alpha_2)$ be totally irrational. Suppose that the hypothesis $H_1$ for the discontinuities of $\varphi$ is satisfied. 
Then the $\R$-valued cocycle $(\varphi_{n, \alpha})$ over the rotation $T_\alpha$ on $\T^2$ is ergodic.
\end{thm}
\proof \ We can suppose $\lambda_1=\lambda_1(\varphi)>0$ w.l.g. Let $\theta_1, \theta_2$ be positive real numbers such that,
with $c$ defined in Lemma \ref{lem:bad},
\begin{eqnarray}
0 < \theta_1 < \theta_2 < \frac{c}{100} \,\lambda_1, \label{hypdelta0}
\end{eqnarray}

Let $B \subset \T^2$ be any measurable set of positive measure. We are going to show that there are infinitely many integers $n \in \N$ such that
\begin{eqnarray}
\mu(B \cap T_{\alpha}^{-n} B  \cap \{|\varphi_n| \in [\theta_1, \theta_2] \}) > 0. \label{thetaInt}
\end{eqnarray}
As $\theta_1, \theta_2$ are arbitrary in $]0, \frac{c}{100} \,\lambda_1[$, this will imply
that $\theta$ or $-\theta$ is an essential value of the cocycle $(\varphi_n)$ for every $\theta\in]0, \frac{c}{100} \,\lambda_1[$.
Since ${\mathcal E}(\varphi)$ is a closed subgroup of $\R$, this shows ergodicity. It remains to prove (\ref{thetaInt}).

{\it Proof of (\ref{thetaInt}).}

The proof is divided into two steps. 
The  first step aims to Inequality (\ref{muRn0}) below on density of subsets with respect to partitions $\cal U_n$ associated with $(\varphi_n)$.
The proof of (\ref{muRn0}) relies on a version of the Lebesgue density theorem adapted to the partitions $\cal U_n$. 
Thanks to recurrence, the second step combines Lemma \ref{lem:var} and (\ref{muRn0}).

{\it 1a) Definition of the partitions $\cal U_n$}

For $\varepsilon>0$, let $\omega(\varepsilon)$ be a modulus of (local) continuity for the function $\varphi$, i.e., if two points $(x_1,x_2)$ and $(y_1,y_2)$ of the torus $\T^2$ 
are in the same rectangle of the partition $\cal P_{\varphi}$ associated with $\varphi$ and if $\max_i |y_i-x_i| \leq \omega(\varepsilon)$, then $|\varphi(x_1,x_2)-\varphi(y_1,y_2)|\leq \varepsilon$.

Let $n\geq 1$. The partition $\cal U_n$ will be a refinement of the partition $\cal P_{\varphi}^n$. 
First, each of the sets $J^i=\{\beta^i_j-k\alpha_i:1\leq j\leq r_i,\, 0\leq k<n\}$, $i=1,2$, cuts $\T^1$ into a set  ${\cal I}^i_n$ of half-open intervals open on the right.  
Next, let $\delta_n=\min\{|I|: I \in {\cal I}^2_n\}$ and let $e_n=\min(\delta_n,\omega(\frac1{2n^2}))$. 
These quantities are non-increasing with $n$.

All the intervals in $\cal I^2_n$ have a length $\geq e_n$. We divide each $I\in {\cal I}^2_n$ into sub-intervals of length $\frac12e_n$, 
except the last sub-interval with a length between $\frac12 e_n$ and $e_n$.
 
We obtain a new set of non-overlapping intervals ${\cal J}^2_n$ such that
\begin{itemize}
\item every interval in $\cal J^2_n$ is included in an interval of $\cal I^2_n$,
\item every interval in $\cal I^2_n$ is  a union of intervals in $\cal J^2_n$,
\item the length of every interval of $\cal J^2_n$ is in $[{e_n \over 2}, e_n[$. 
\end{itemize}

In the sequel of the proof, we write simply $T$ for the rotation $T_\alpha$ on $\T^2$. We define 
$$\cal U_n=\{R := I\times J:I\in \cal I^1_n,\, J\in \cal J^2_n\}.$$
We will also use the partition $\tilde{\cal U}_n=\{\tilde R := T^n R: R \in \cal U_n\}$. We denote $R_n(x)$ (resp. $\tilde R_n(x)$) 
the rectangle of $\cal U_n$ (resp. $\tilde{\cal U}_n$) that contains $x \in\T^2$ (well defined for $x$ outside a set of 0 measure).

{\it 1b) Application of Lebesgue density theorem}

We want to use a Lebesgue density theorem  (see appendix, Theorem \ref{densL}) twice, once  with the partitions $\cal U_n$, $n\geq 1$, and once 
with the partitions $\tilde{\cal U}_n$, $n\geq 1$. We need to check that Conditions (\ref{condi1bis}) and (\ref{condi2}) of Theorem \ref{densL}) hold 
for both sequences of partitions. Condition (\ref{condi2}) about the diameters is clearly satisfied.

Next, let $L_n :=\max\{|I|:I\in\cal I^1_n\}$. The sequence $(L_n)$ is non-increasing and, by Lemma \ref{lem:bad}, 
$L_n \leq C\ell_n$, where $\ell_n=\min\{|I|:I\in\cal I^1_n\}$ and  $C=\frac{c'}{c}$. 
Let $I\times J\in \cal U_n$ or $\tilde{\cal U}_n$. On the one hand, $\mu(I\times J)\geq \ell_ne_n/2$. On the other hand, if a product of intervals, $I'\times J'$, 
with lengths respectively $\leq L_k$ and $\leq e_k$ for some $k\geq n$, intersects $I\times J$, then
$$I'\times J'\subset I''\times J''=(I+[-L_n, L_n])\times (J+[-e_n,e_n]).$$
Since $\mu(I''\times J'')\leq 3L_n\times 3e_n\leq 9C\mu(I\times J)$, (\ref{condi1bis}) holds.

Thanks to the Lebesgue density theorem  for the set $B$ with respect to the two families of rectangles ${\cal U}_n$ and $\tilde {\cal U}_n$,
we have:
$$\lim_{n \to \infty} {\mu(R_n(x) \cap B) \over \mu(R_n(x))}
= \lim_{n \to \infty} {\mu(\tilde R_n(x) \cap B) \over \mu(\tilde R_n(x))} = 1, \text {for a.e. } x \in B.$$

Therefore, for all $\varepsilon> 0$, there exist an integer $n_\varepsilon$ and a subset $B_\varepsilon$ of $B$ of positive measure such that, 
for all $u \in B_\varepsilon$ and all $n \geq n_\varepsilon$,
$${\mu(R_n(u) \cap B) \over \mu(R_n(u))} \geq 1 - \varepsilon, \ {\mu(\tilde R_n(u) \cap B) \over \mu(\tilde R_n(u))} \geq 1 - \varepsilon.$$

Let $n \geq n_\varepsilon$ and let $u$ be in $B_\varepsilon \cap T^{-n}B_\varepsilon$.

Since $u \in B_\varepsilon$, the first previous inequality implies
$$\mu(B \cap R_n(u)) \geq (1 - \varepsilon ) \mu(R_n(u)).$$
Since $T^n u \in B_\varepsilon$, the second inequality implies
$$\mu(B \cap \tilde R_n(T^n u)) \geq (1 - \varepsilon ) \mu(\tilde R_n(T^n u)).$$
Now, $T^{-n} \tilde R_n(T^n u) = R_n(u)$ and $T^{-n} (B \cap \tilde R_n(T^n u)) = T^{-n} B \cap R_n(u)$.
Hence
$$\mu(T^{-n} B \cap R_n(u)) \geq (1 - \varepsilon ) \mu(R_n(u)).$$
In brief, for every $\varepsilon >0$, there are $n_\varepsilon \geq 1$ and a set of positive measure $B_\varepsilon\subset B$ such that
\begin{eqnarray}
&&\mu(B \cap T^{-n} B \cap R_n(u)) \geq (1 - 2 \varepsilon ) \mu(R_n(u)),  
 \label{muRn0}
\end{eqnarray}
for all $n \geq n_\varepsilon $ and all $ u \in B_\varepsilon \cap T^{-n} B_\varepsilon$.
It follows, for $n \geq n_\varepsilon$,
\begin{eqnarray}
\mu((B \cap T^{-n} B )^c \cap R_n(u)) \leq 2 \varepsilon \mu(R_n(u)), \forall u \in B_\varepsilon \cap T^{-n} B_\varepsilon. \label{BBRn}
\end{eqnarray}

With $c'$ defined in Lemma \ref{lem:bad}, we take $\varepsilon$ such that 
\begin{eqnarray}
0 < \varepsilon <  \frac {\theta_2-\theta_1}{32c'\lambda_1}. \label{epsil}
\end{eqnarray}

{\it 2) Application of recurrence}

Let $N$ be the integer defined in Lemma  \ref{lem:var}.
As the function $\varphi$ is centered, the cocycle $(\varphi_n)$ is recurrent. Therefore, by Remark \ref{rem1}.a), 
there exists $n \geq \max(n_\varepsilon, N)$ and $a = (a_1, a_2) \in B_\varepsilon \cap T^{-n} B_\varepsilon$ such that $|\varphi_n(a)| <\frac12 \theta _1$.

Let $R_n(a) = [s_1, t_1] \times [s_2, t_2]\in \cal U_n$ be the rectangle of $\cal U_n$ that contains $a $. 
By Lemma \ref{lem:bad},  $\frac{c}{n}\leq t_1-s_1\leq \frac{c'}{n}$.

The real number $a_1$ is either in the first half of the interval $[s_1, t_1]$, or in the second half of this interval. Suppose that $a_1$ in the
first half of this interval (if it is in the second half, just move in the negative direction instead of the positive direction). 

Let $x_2\in[s_2,t_2]$ and consider the function $f_{x_2}:t\in[0,t_1-a_1[\rightarrow\varphi_n(a_1+t,x_2)$. 
We want to bound from below the length of the set of $t$ such that $f_{x_2}(t)\in[\theta_1,\theta_2]$.
By Lemma \ref{lem:var}, since $n\geq N$, for all $t\in[0,t_1-a_1]$ the derivative $f'_{x_2}(t)$ is positive and 
\begin{eqnarray*}
\tfrac{1}{2}n\lambda_1 t\leq f_{x_2}(t) -f_{x_2}(0) \leq 2n\lambda_1 t.
\end{eqnarray*}
By the definitions of $\cal J^2_n$, of $\omega(\frac1{2n^2})$ and $e_n$, we have 
$|f_{x_2}(0)|\leq |\varphi_n(a_1,a_2)|+n\times\frac1{n^2}\leq \tfrac12\theta_1+\tfrac1n$, so that $|f_{x_2}(0)|\leq \theta_1$ provided that $n\geq\frac2{\theta_1}$. 
Using (\ref{hypdelta0}), we also have 
\begin{eqnarray*}
&f_{x_2}(\tfrac12(t_1-s_1))\geq f_{x_2}(0)+\tfrac12n\lambda_1\times \tfrac12(t_1-s_1)
\geq -\theta_1+\tfrac14n\lambda_1\frac{c}{n}\geq -\theta_1+25\theta_2\geq \theta_2.
\end{eqnarray*}
It follows that $[\theta_1,\theta_2]\subset f_{x_2}([0,\frac12(t_1-s_1)]$ which in turn implies that
$$|\{t\in[0,\tfrac12(t_1-s_1)]:\varphi_n(a_1+t,x_2)\in[\theta_1,\theta_2]\}|
\geq \frac{\theta_2-\theta_1}{\max\{|f'_{x_2}(t):t\in[0, (t_1-a_1)]\}}\geq \frac{\theta_2-\theta_1}{2n\lambda_1}.$$ 
By Fubini's theorem, the set $\displaystyle A := \{(x_1, x_2) \in R_n(a) : \varphi_n(x_1,x_2)\in[\theta_1,\theta_2] \}$
has a measure $\displaystyle \mu(A) \geq  \frac{\theta_2-\theta_1} {2n\lambda_1}\times \frac{e_n }{2}$. This implies by (\ref{epsil}): 
\begin{eqnarray}
&&{\mu(A) \over \mu(R_n(a))} \geq {\frac{\theta_2-\theta_1} {2n\lambda_1}\times \frac{e_n }{2} \over \frac {c'} {n}\times e_n} 
= \frac {\theta_2-\theta_1}{4c'\lambda_1} \geq 4 \varepsilon. \label{muA}
\end{eqnarray}
By definition of $A$, $\varphi_n(A) \subset [\theta_1, \theta_2]$. 
As $A \subset R_n(a)$, it follows, using (\ref{BBRn}) and (\ref{muA}):
\begin{eqnarray*}
&&\mu(B \cap T^{-n} B \cap  \varphi_n^{-1} [\theta_1, \theta_2]) \geq \mu(B \cap T^{-n} B \cap A) \\
&&= \mu(A) - \mu((B \cap T^{-n} B )^c \cap A) \geq \mu(A) - \mu((B \cap T^{-n} B )^c \cap R_n(a)) \\
&&\geq \mu(A) - 2 \varepsilon \mu(R_n(a)) \geq 4 \varepsilon \mu(R_n(a)) - 2 \varepsilon \mu(R_n(a)) = 2 \varepsilon \mu(R_n(a))  > 0. 
\end{eqnarray*}
This shows (\ref{thetaInt}). \eop

\subsection{Class ${\cal F}_2$}

\

Now we consider a class ${\cal F}_2$ of functions from $\T^2$ to $\R^2$ whose components belong to the class ${\cal F}_1$,
but with a stronger regularity condition in both variables $x_1, x_2$.
\begin{definition} \label{defF2} {\rm ${\cal F}_2$ is the class of centered functions $\varphi = (\varphi^1, \varphi^2)$ on $\T^2$ such that, for two finite sets $J_1, J_2$ 
depending on $\varphi$: $J_i= \{\beta^i_0=0 \leq \beta^i_1 < ... < \beta^i_{r_i-1} \leq  \beta^i_{r_i}= 1\}, \,  i=1, 2$,
$\varphi$ is $C^1$ on the elements of the partition ${\cal P} = {\cal P}_\varphi$ of $[0, 1] \times [0, 1]$ into the open rectangles 
$\displaystyle P_{j, j'} = ]\beta^1_j, \beta^1_{j+1}[ \, \times \, ]\beta^2_{j'}, \beta^2_{j'+1}[, \, j=0,..., r_1-1, j'=0, ..., r_2-1$. 

Moreover, we assume that the partial derivatives of $\varphi^i$, $i=1, 2$, can be extended to continuous functions on the closure of the elements of ${\cal P}_\varphi$.}
\end{definition}
The following quantities are associated to a function $\varphi  =  (\varphi^1, \varphi^2) \in {\cal F}_2$: 
\begin{eqnarray*}
\lambda_1(\varphi^i) := \int_{\T^2} {\partial \varphi^i \over \partial x_1} \, dx, 
\ \lambda_2(\varphi^i) := \int_{\T^2} {\partial \varphi^i \over \partial x_2} \, dx, \ i=1,2.
\end{eqnarray*}
As in Lemma \ref{lem:var}, by unique ergodicity of the rotation and since ${\partial \varphi^i \over \partial x_1}$ and ${\partial \varphi^i \over \partial x_2}$ 
are Riemann integrable, we have uniformly for $x = (x_1, x_2) \in \T^2$:
\begin{eqnarray}
&&\lambda_j(\varphi^i) = \lim_n \frac1n \sum_{k=0}^{n-1}{\partial \varphi^i \over \partial x_j}(T^k(x_1,x_2)), \ i, j = 1, 2. \label{deriv1}
\end{eqnarray}

{\bf Hypothesis $H_2$ on $\alpha$ and the discontinuities $\beta^i_j $}:
$$\beta^i_j - \beta^i_{j'} \in \operatorname{Bad}_{\Z}(\alpha_i), \, \forall j ,j' \in \{1, \dots, r_i\},  \text{ for } i=1,2.$$ 
\begin{rem} 
In particular, $\alpha_1, \alpha_2$ are in Bad (cf. Remark \ref{rkH} a)). It will be shown in Subsection \ref{Badly} that 
\begin{itemize}
\item given $\alpha$, the set of $(\beta_j^i,  \, j= 1, ..., r_i, i=1,2)$ satisfying $(H_2)$ has Hausdorff dimension $r_1+r_2$ in $\R^{r_1+r_2}$;
\item given $(\beta^i_j, j=1,\dots,r_i, i=1,2)$, the set of $\alpha$ such that $(H_2)$ hold has Hausdorff dimension 2.
\end{itemize}
\end{rem}

\vskip 3mm
\begin{ex} \label{Gamma} {\rm  Let $\gamma_1 < \gamma_2 \in [1, 2[$. Consider the function 
$\varphi_{\gamma_1, \gamma_2}(x_1,x_2) =\{\gamma_1 x_1\} \{\gamma_2 x_2\}$ 
restricted to $[0, 1[ \times [0, 1[$. By an elementary computation, we obtain 
\begin{eqnarray*}
&&\int_0^1 \int_0^1 \varphi_{\gamma_1, \gamma_2} \, dx_1 dx_2 = \mu(\varphi_{\gamma_1, \gamma_2}) 
= (\frac12 \gamma_1 -1 + \gamma_1^{-1})  (\frac12 \gamma_2 -1 + \gamma_2^{-1}),\\
&&\lambda_1(\varphi_{\gamma_1, \gamma_2}) = \gamma_1 (\frac12 \gamma_2 - 1 + \gamma_2^{-1}),
\ \lambda_2(\varphi_{\gamma_1, \gamma_2}) = \gamma_2 (\frac12 \gamma_1 - 1 + \gamma_1^{-1}).
\end{eqnarray*}
Putting $\varphi^1= \varphi_{\gamma_1, \gamma_2} - \mu(\varphi_{\gamma_1, \gamma_2}) , \varphi^2= \varphi_{1, 1} - \frac14$, 
the function $\Phi := (\varphi^1, \varphi^2)$ is in ${\cal F}_2$.

The sets $J_1, J_2$ of the definition of ${\cal F}_2$ in \ref{defF2} are:
$$J_1= \{\beta^1_0=0 < \beta^1_1 = \gamma_1^{-1} < \beta^1_2= 1\}, \ J_1= \{\beta^2_0=0 < \beta^2_1 = \gamma_2^{-1} < \beta^2_2= 1\}.$$

If $\alpha_1$ and $\alpha_2$ are in Bad, the condition $H_2$ is satisfied if $\gamma_i$ is rational 
or more generally if $\gamma_i^{-1} \in \operatorname{Bad}_{\Z}(\alpha_i)$, $i=1,2$. 
Moreover, if $\gamma_1 \not = \gamma_2$ and ${\gamma_1 + \gamma_2 \over  \gamma_1  \gamma_2} \not = 1$, it holds
$$\lambda_1(\varphi_{\gamma_1, \gamma_2}) \lambda_2(\varphi_{1, 1}) - \lambda_2(\varphi_{\gamma_1, \gamma_2}) \lambda_1(\varphi_{1, 1}) 
= \frac12 (\gamma_2 - \gamma_1) (1 - {\gamma_1 + \gamma_2 \over  \gamma_1  \gamma_2}) \not = 0.$$ } \end{ex}

We use the notation of the previous subsection (cf. (\ref{Rll}).
For every $n \geq 1$, the ergodic sums $\varphi_n$ are $C^1$ on the atoms 
$R_{\ell, \ell'}^n = ]\gamma_{n,\ell}^1, \gamma_{n,\ell+1}^1[ \, \times \, ]\gamma_{n,\ell'}^2, \gamma_{n,\ell'+1}^2[$ of the partition 
$\displaystyle {\cal P}_{\varphi}^{n} := {\cal P} \wedge T_\alpha^{-1}{\cal P} \wedge ... \wedge T_\alpha^{n-1}{\cal P}$.
We consider also the partition $\tilde {\Cal P}_n= T_{\alpha}^n {\Cal P}_n$.

We will use the following variant of Lemma \ref{lem:var}:
\begin{lem}\label{lemdev} Let $\varphi = (\varphi^1, \varphi^2)$ be in $\cal F_2$. If $x = (x_1, x_2)$ and $x+u = (x_1+u_1, x_2+u_2)$ belong 
to the same element of the partition ${\cal P_{\varphi}^n}$, we have,
\begin{eqnarray}
&&\varphi^i_n(x+u) = \varphi^i_n(x) + (\lambda_1(\varphi^i) u_1 + \lambda_2(\varphi^i) u_2) \, n + o(n) |u| + \varepsilon(u) |u| \, n, \label{dev}
\end{eqnarray}
with $\varepsilon(t)$, defined for $|t|$ small, depending only on $\varphi$ and such that $\lim_{t \to 0} \varepsilon(t) = 0$.
\end{lem}
\proof \  By the hypothesis on $x$ and $x+u$, their images $T_\alpha^k x$, and $ T_\alpha^k (x+u)$, for $0 \leq k < n$,
 belong to the same rectangle $P_{j, j'}$ for some $j,j'$. For $i=1,2$, as the partial derivatives of $\varphi^i$ can be extended to continuous functions on the closure of $P_{j, j'}$, 
 by the mean value theorem, we have, with $\varepsilon$ as in the statement:
\begin{eqnarray*}
&&|\varphi^i(T_\alpha^k (x+u)) - [\varphi^i(T_\alpha^k x)
+ u_1 {\partial \varphi^i \over \partial x_1}(T_\alpha^k x) + u_2 {\partial \varphi^i \over \partial x_2}(T_\alpha^k x)]| 
\leq \varepsilon(u) |u|.
\end{eqnarray*}
It follows: 
\begin{eqnarray*}
&&|\varphi^i_n(x+u) - [\varphi^i_n(x)+ u_1 \sum_{k=0}^{n-1}{\partial \varphi^i \over \partial x_1}(T_\alpha^k x) 
+ u_2 \sum_{k=0}^{n-1} {\partial \varphi^i \over \partial x_2}(T_\alpha^k x)]| \leq \varepsilon(u) |u| \,n.
\end{eqnarray*}
Using (\ref{deriv1}), we get (\ref{dev}). \eop

\subsection{Ergodicity of $(\varphi_n)$ for $\varphi =(\varphi^1, \varphi^2) \in {\cal F_2}$}

\

\begin{thm}\label{xymod} Let $\varphi=(\varphi^1, \varphi^2) \in {\cal F_2}$ be such that 
$\lambda_1(\varphi^1) \lambda_2(\varphi^2) - \lambda_2(\varphi^1) \lambda_1(\varphi^2) \not = 0$.

A) Let $\alpha = (\alpha_1, \alpha_2)$ be totally irrational. Suppose that the hypothesis $H_2$ for the discontinuities of $\varphi^1, \varphi^2$ is satisfied. 
If the $\R^2$-valued cocycle $(\Phi_{n, \alpha})$ generated by $\Phi = (\varphi^1, \varphi^2)$ over the rotation  $T_\alpha$ on $\T^2$ is recurrent,
then it is ergodic.

B) Suppose that $\varphi^1, \varphi^2$ have bounded partial derivatives of first and second order on the interior of their continuity domain.
Given the discontinuities $\beta_j$, the set of $\alpha\in\R^2$ such that the cocycle $(\Phi_{n, \alpha})$ is ergodic is of Hausdorff dimension $2$. 
\end{thm}
\proof \  A) We keep the notations of Section \ref{sec:F}. By Lemma \ref{lem:bad}, Hypothesis $H'_2$ implies the existence of two positive constants $c, c'$ such that
\begin{equation}\label{defcs1}
{\frac{c}{n}} \leq \gamma_{n, \ell+1}^i - \gamma_{n,\ell}^i \leq \frac{c'}{ n}, \ \ell= 1,\ldots, |J^i|n, \ i=1,2.
\end{equation}
For $x \in [0,1] \times [0,1] \setminus \bigcup_k \partial {\cal P}_{\varphi}^k$, let $R_n(x)$ be the rectangle element of the partition ${\Cal P}_{\varphi}^n$ containing $x$.
Likewise let $\tilde R_n(x)$ be the element of the partition $\tilde {\Cal P}_{\varphi}^n$ containing $x$.

Let $x=(x_1, x_2)$ and $u=(u_1,u_2)$ be such that $x$ and $x+u$ belong to the same element of the partition ${\cal P_{\varphi}^n}$.
It follows from Lemma \ref{lemdev} for $n$ big and $|u|$ small: 
\begin{eqnarray*}
\varphi_n^i(x+u) = \varphi_n^i(x) + (\lambda_1(\varphi^i) u_1 + \lambda_2(\varphi^i) u_2)n + o(n) |u|   + o(|u|) n, \, i=1,2.
\end{eqnarray*}
We set $M =  \left( \begin{matrix} \lambda_1(\varphi^1) & \lambda_2(\varphi^1)& \cr \lambda_1(\varphi^2) & \lambda_2(\varphi^2) \end{matrix} \right)$ and
$u^t =  \left( \begin{matrix} u_1 \cr u_2 \end{matrix} \right)$. Since $|u|=O(\tfrac1n)$ by (\ref{defcs1}), there is for every $\varepsilon >0$
an integer $N^1_\varepsilon$ depending only on $\varphi$ and $\varepsilon$ such that
\begin{eqnarray}
|\varphi_n(x+u) - [\varphi_n(x) + n M u^t]| \leq \varepsilon, \text{ for } n > n_\varepsilon. \label{dev1}
\end{eqnarray}
Therefore, for $n >  N^1_\varepsilon$, 
\begin{eqnarray}
|\varphi_n(x+u) - n M u^t| \leq \varepsilon + |\varphi_n(x)|. \label{MatrApprox}
\end{eqnarray}
Recall that $c$ and $c'$ are defined in (\ref{defcs1}). The matrix $M$ is invertible by  the hypothesis of the theorem.
For $c_0 = \frac12 c$, the image by $M$ of a square $Q = [0, c_0] \times [0, c_0]$ is a parallelogram ${\cal L}_0$ 
with a vertex at the origin.

Suppose that ${\mathcal E}(\varphi) \not = \R^2$. According to the form of the closed subgroups of $\R^2$, there is an open ball $U$ of radius $r > 0$ in $\R^2$,
such that its closure, the compact set $K =\overline U$, is contained in ${\cal L}_0$ and is disjoint from ${\mathcal E}(\varphi)$.
For a constant $\nu >0$, we have $\lambda(U) \geq \nu \lambda({\cal L}_0)$. Calling $U_0 \subset U$ the ball of radius $r-2\varepsilon$ with the same center as $U$,
we take $\varepsilon$ small enough so that $\lambda(U_0) \geq \frac12 \nu \lambda({\cal L}_0)$ and $ \frac12 \nu (\frac{c_0}{c'})^2> 2 \varepsilon$.

The theorem will be proved by contradiction if we show:

{\it Claim: $K$ contains an essential value of $\varphi$. }

{\it Proof of the claim.} 

Suppose that $K \cap {\mathcal E}(\varphi) = \emptyset$. To get a contradiction, we use Lemma \ref{compact} which implies that there exists 
$B \in\cal{B}$ such that $\mu(B)>0$ and 
\begin{eqnarray}
\mu(B\cap T^{-n}B \cap (\varphi_n\in K))=0, \, \forall n \in \Z. \label{meas0}
\end{eqnarray}
As in Theorem \ref{thm:ergodicity1}, using the Lebesgue density theorem (Theorem \ref{densL}) for $B$, we obtain: 

For every $\varepsilon >0$, there are $N^2_\varepsilon \geq 1$ and a set of positive measure $B_\varepsilon\subset B$ such that
\begin{eqnarray}
&&\mu(B \cap T^{-n} B \cap R_n(u)) \geq (1 - 2 \varepsilon ) \mu(R_n(u)), \forall u \in B_\varepsilon \cap T^{-n} B_\varepsilon, 
\text{ for } n \geq N^2_\varepsilon. \label{muRn}
\end{eqnarray}

The recurrence of $(\varphi_n)$ implies (cf. Remark \ref{rem1}.a) that there exist infinitely many integers $n \geq N^2_\varepsilon$ such that
\begin{eqnarray}
&&\mu(B_\varepsilon\cap T_\alpha^{-n} B_\varepsilon \cap (|\varphi_n|< \varepsilon)) > 0. \label{apprec}
\end{eqnarray}
Therefore, we can choose $n\geq \sup(N^1_{\varepsilon}, N^2_{\varepsilon})$ and $x^0 \in \T^2$ 
such that $|\varphi_n(x^0)|< \varepsilon$ and $x^0 \in B_\varepsilon \cap T^{-n} B_\varepsilon$,
which implies by (\ref{muRn}):
\begin{eqnarray}
\mu(B \cap T^{-n} B \cap R_n(x^0)) \geq (1 - 2 \varepsilon ) \mu(R_n(x^0)). \label{muRx0}
\end{eqnarray}
We can assume that $x^0$, which belongs to $R_n(x^0)$, is one of the corners of a square $Q_n\subset R_n(x^0)$ of size ${c_0 \over n} \times {c_0 \over n}$. 
Up to a change of the signs of $u_1, u_2$, we can also assume that $x^0$ is the lower left corner.

According to the definition of $\Cal L_0$ at the beginning of the proof, $nM(Q_n-x^0) = {\cal L}_0$ and the measure of the open set $W = (nM)^{-1} U_0 + x^0$ 
satisfies $\mu(W)=\mu(Q_n)\frac{\lambda(U_0)}{\lambda({\cal L}_0)} \geq \frac12 \nu(\frac{c_0}{c'})^2 \mu(R_n(x^0))$ by (\ref{defcs1}). 
Recall that $\varepsilon$ is such that $\frac12 \nu(\frac{c_0}{c'})^2 -2\varepsilon > 0$.

Clearly, $W\subset Q_n\subset R_n(x^0)$. Moreover,  since $|\varphi_n(x^0)|< \varepsilon$,  by (\ref{MatrApprox}) we have 
$$W = \{x: nM (x - x^0) \in U_0\} \subset \{x: d(\varphi_n(x), U_0) \leq 2 \varepsilon\} \subset \{x: \varphi_n(x) \in U\} \subset \{x: \varphi_n(x) \in K\}.$$
Observe that for any three sets $E_1, E_2, E_3$, 
\begin{eqnarray*}
\mu(E_1 \cap E_2 \cap E_3) \geq \mu(E_1 \cap E_2) - \mu(E_2) + \mu(E_2 \cap E_3). 
\end{eqnarray*}
Using (\ref{muRx0}), it follows, with $E_1= B \cap T^{-n} B$, $E_2 =  R_n(x^0)$, $E_3 = W$:
\begin{eqnarray*}
&&\mu(B \cap T^{-n} B \cap R_n(x^0) \cap \{x: \varphi_n(x) \in K\}) \geq \mu(B \cap T^{-n} B \cap R_n(x^0) \cap W)\\ 
&&\geq (1 - 2 \varepsilon) \, \mu(R_n(x^0)) - \mu(R_n(x^0)) +\mu(R_n(x^0) \cap W) \geq (\frac12 \nu(\frac{c_0}{c'})^2 -2\varepsilon) \, \mu(R_n(x^0)) > 0. 
\end{eqnarray*}
This gives a contradiction with (\ref{meas0}) and concludes the proof of A). 

B) If $(\varphi^1, \varphi^2)$ is in $\cal F$ and also satisfy the hypothesis of B), then $\varphi^1$ and $\varphi^2$ are in $\cal G$ (cf. Definition \ref{classG}).

By Theorem \ref{thm:recurrence}, the set of $\alpha\in\R^2$ such that the cocycle $(\varphi_n)$ is not recurrent has a Hausdorff dimension $\leq 2- 1/3$.  
Let $B$ be its complement (the set of $\alpha$ such that $(\Phi_n)$ is recurrent).
The set $A$ of $\alpha$ such that the condition $H_2$ is satisfied has a Hausdorff dimension 2 by Corollary \ref{cor:badApprox2}. 
Therefore the Hausdorff dimension\footnote{For the sets A and B, we have
$2 = \dim_H A \leq \max(\dim_H(A \cap B), \dim_H(A \cap B^c)) \leq \max(\dim_H(A \cap B), \dim_H(B^c)) \leq \max(\dim_H(A \cap B),  2- 1/3)$;
hence $\dim_H(A \cap B) = 2$.} of $A \cap B$ is 2.

By A) it follows that the set of $\alpha \in \R^2$ such that the cocycle $(\varphi_n)$ is ergodic, is of Hausdorff dimension $2$. 
\eop

\vskip 3mm
{\it An algebraic example}: If $\alpha_1, \alpha_2$ are quadratic, then recurrence follows from Theorem \ref{thm:recurrence} 2).
As $\alpha_1, \alpha_2$ are in Bad, we obtain ergodicity for the cocycle defined in Example \ref{Gamma}.
This gives an explicit example of an ergodic cocycle.

\section{\bf Ergodicity of compact extensions for the triangle $\Delta_0$} \label{Delta0}

In Theorems \ref{thm:ergodicity1} and \ref{xymod}, the discontinuities of the function $\Phi:\T^2\fff\R^d$ generating an ergodic cocycle 
lie along lines parallel to the coordinate axes. It is not easy to adapt our method to construct cocycles generated by a function with more general discontinuities. 

When the set of discontinuities is the boundary of the triangle $\Delta_0 = \{(x_1,x_2) \in [0,1]^2: x_1>x_2\}$,  
the diameters of the connected components of the continuity set of the ergodic sum $\Phi_n$, vary at least from $ 1/n^\gamma$ to $1/n$ 
for arbitrarily large values of $n$ where $\gamma=\frac{1+\sqrt 5}2$ is the golden ratio. 

Specifically, the distance from the discontinuity line $x_1 = x_2 \mod \Z$ to $(-n_1\alpha_1,-n_2\alpha_2)$, 
intersection of the vertical discontinuity line  $x_1=-n_1\alpha_1\mod\Z$ and of the horizontal discontinuity line  
$x_2=-n_2\alpha_2\mod \Z$,  is $\frac1{\sqrt2}\|n_1\alpha_1-n_2\alpha_2\|$. 
Moreover, according to a result by W.M. Schmidt \cite{WSch70}, if $(\alpha_1, \alpha_2)$ is totally irrational, there exist infinitely many integers $n$ 
such that there are integers $0<n_1,n_2<n$ with $\|n_1\alpha_1-n_2\alpha_2\|\leq n^{-\gamma}.$  

This large ratio between small and large diameters implies that Lebesgue's density theorem cannot be used directly. 
Furthermore, the diameter of the  image by $\Phi_n$ of a small component is  $\leq n\times 1/n^\gamma$, which is insufficient to establish the existence of non-zero essential values.

Nevertheless, some partial results for $\Delta_0$ can be shown, notably the ergodicity of compact extensions.

{\bf Ergodicity of compact extensions for $1_{\Delta_0}$, $\Delta_0 = \{(x, y) \in [0, 1]^2: x < y \}$.}

For a general compact extension $\tilde T_\varphi: (x, y) \to (Tx, y + \varphi(x))$, where $T$ is an ergodic measure preserving transformation on a probability space 
$(X, {\cal B}, \mu)$ and $\varphi$ a measurable function on $X$ with values in $\T^d$, $d \geq 1$, one obtains easily a criterium of ergodicity:
\begin{lem} $\tilde T_\varphi$ is ergodic on $X \times \T^d$ if and only if the functional equation
\begin{eqnarray}
H(Tx) = e^{2\pi i  \langle \k,\varphi(x) \rangle} \, H(x) \label{equa0}
\end{eqnarray}
has no measurable solution $H$ of modulus 1 for $\k \in \Z^d \setminus \{\0\}$. 
\end{lem}
\proof \ Let $F: X \times\T^d\fff\R$ be a measurable $\tilde T_{\varphi}$-invariant function. By truncation, we can suppose $F$ bounded. 
For $\k \in \Z^d\setminus\{\0\}$, its Fourier coefficient with respect to $y$ satisfies  
\begin{eqnarray*}F_{\k}(x)=\int_{\T^d} F(x,y) e^{-2\pi i \langle \k,y \rangle} dy
=\int_{\T^d} F(\tilde T_{\varphi}(x,y)) e^{-2\pi i \langle \k,y \rangle} dy=e^{2\pi i  \langle \k,\varphi(x)\rangle} F_{\k}(T x).
\end{eqnarray*}
Therefore, $|F_k|$ is a.e.-constant and must be a.e. zero by the sufficient condition of the lemma.\eop

Now we take $\a = (a_1, ... a_d) \in \R^d$, put $\Phi_{\a} = 1_{\Delta_0} \,  \a$ 
and consider the skew-product on $\T^2 \times \T^d$ defined for $\alpha = (\alpha_1, \alpha_2)$ by
$$\tilde T_{\alpha, \Phi_{\a}}: (x, y) \to (x+\alpha \modu, y+ \Phi_{\a}(x) \modu).$$
Our aim in this section is to prove that $\tilde T_{\alpha, \Phi_{\a}}$ is ergodic, when $\alpha$ and $\a$ are totally irrational 
and $\alpha$ satisfies a Diophantine condition. We start by a definition and an auxiliary proposition.

{\it Coding of a map}: Given a space $X$, a map $T:X \to X$ and a partition $\cal Q$ of $X$, for an integer $n\geq 1$ the {\it $({\cal Q}, n)$-coding} of $x\in X$ 
associated with $T$ and $Q$, is the sequence $(w_i, i=0,\dots,n-1)\in \cal Q^n$ such that $T^{i}(x)\in w_i$ for $i=0,\dots,n-1$.

\begin{proposition} \label{eqfunct} Let $(X, {\cal B}, \mu, T)$ be a dynamical system, where $X$ is a compact
metric space with a distance $d_X$, ${\cal B}$ is borelian 
and $T$ is an isometry which preserves $\mu$. Let $G$ be a group endowed with a bi-invariant distance $\delta$
and $\varphi$ a measurable function from $X$ to $G$ taking a finite number of values and non constant a.e. 
Let $\delta_0=\delta_0(\varphi):=\min\{\delta(g,g'): g, g' \text{ distinct values of }\varphi\} >0$. 

Let ${\cal Q}$ be the partition of $X$ into the sets on which $\varphi$ is constant. Suppose that for each integer $\ell \geq 1$, 
there exists a finite partition\footnote{up to a set of 0 measure} $\cal P_{\ell}$ of $X$
such that for each $P \in \cal P_{\ell}$, there exists $m_{\ell}(P) \in \cal Q^{\ell}$ which is the $({\cal Q}, \ell)$-coding of all $x\in P$.

Assume also that there is a sequence $(\ell_k)_{k\geq 1}$ such that

$1)$ $\displaystyle \lim_{k\fff\infty} \max_{P \in {\cal P}_{\ell_k}} {\rm diam}(P) = 0$;
\hfill \break $2)$ for some $\rho \geq 1$, for all $k \geq 1$ and all $P\in {\cal P}_{\ell_k}$, there are at most $\rho$ elements $P' \in\cal P_{\ell_k}$ 
such that the closures satisfy $\overline P \cap \overline {P'} \neq \emptyset$;

$3)$ there is a family ${\cal C}_ {c,\ell_k}\subset \cal P_{\ell_k}$ and positive constants $c$ and $\lambda$ such that 
\hfill\break $3a)$ for all $P\in{\cal C}_{c,\ell_k}$, $ \mu(P) \geq {c \over \Card({\cal P}_{\ell_k})}$;
\hfill \break $3b)$ $N_k := \Card({\cal C}_{c,\ell_k}) \geq \lambda \, \Card({\cal P}_{\ell_k})$;
\hfill \break $3c)$ for all $P \in {\cal C}_{c,\ell_k}$, there is at least one element $P'$ in ${\cal C}_{c,\ell_k}$ such that 
$\overline P\cap \overline {P'} \neq \emptyset$ and the codings $m_{\ell_k}(P)$ and $m_{\ell_k}(P')$ have exactly one different component.

Then there is no measurable solution $f: X \to G$ of the functional equation
\begin{eqnarray} 
\varphi(x) = (f(x))^{-1} \, f(T x), \mu \text{-a.e.} \label{eqfonc0}
\end{eqnarray}
\end{proposition} 
\Proof \ We act by contradiction and suppose that there is a measurable function $f:X\fff G$ such that $\varphi(x) = (f(x))^{-1} \, f(T x), \mu-a.e.$

Let $\varepsilon$ be such that $\displaystyle 0 <  2 \varepsilon < {c \lambda \over 1+ \rho}$. 

There exists a closed set $F$ in $X$ of measure $>1 - \varepsilon$ such that the restriction $f_{|F}$ (hence also $f^{-1}_{|F}$) is uniformly continuous. 
Let $\eta > 0$ be such that the conditions $x,y \in F$, $\dd_X(x,y) < \eta$ imply $\delta(f(x), f(y))$ and $\delta((f(x))^{-1}, (f(y))^{-1})< \delta_0/2$. 

By 1), there exists $k_0$ such that, for all $k \ge k_0$ and all $P \in \cal P_{\ell_k}$, ${\rm diam}(P) < \eta/2$. For $k \ge k_0$, 
if $P$ and $P'$ are  in $\cal P_{\ell_k}$ and such that $\overline P\cap\overline P'\neq\emptyset$, we have then $d(x,y) \leq \eta$ for any $x \in P$ and $y \in P'$. 

Let $H := F \cap T^{-\ell_k} F$ and ${\cal B}_{c,\ell_k}:=\{P\in {\cal C}_{c,\ell_k}:P\cap H\neq\emptyset\}$. Let $N_{1,k} := \Card({\cal B}_{c,\ell_k})$.

Let $P \neq P' \in{\cal P}_{\ell_k}$ be such that $\overline P\cap \overline{P'} \neq \emptyset$ and $m_{\ell_k}(P)=(u_0,\dots,u_{\ell_k-1})$ 
and $m_{\ell_k}(P')=(u'_0,\dots,u'_{\ell_k-1})$ have exactly one different component.
We claim that $P$ and $P'$ cannot both intersect $H$. 

For each $Q \in\cal Q$, denote $g_Q$ the constant value of $\varphi$ on $Q$. 
Since $\varphi_{\ell_k}(x)$ depends only on the coding of $x$, $\varphi_{\ell_k}$ is constant on $P$ and on $P'$, 
and these constants are the products $\pi_P=g_{u_0}\dots g_{u_{\ell_k-1}}$ and $\pi_{P'}=g_{u'_0}\dots g_{u'_{\ell_k-1}}$. 
Since $u_i=u'_i$ for all $i$ except for one $i$, say $i_0$, the bi-invariance of the distance $\delta$ implies that 
$\delta(\pi_P,\pi_{P'})=\delta(g_{u_{i_0}},g_{u_{i_0}})\geq \delta_0$.

Now by assumption we have $\varphi=f^{-1}f\circ T$, hence $\varphi_{\ell_k}=f^{-1}f\circ T^{\ell_k}$. If $x\in P$ and $y\in P'$,
we have $\delta(x,y)\leq \eta$ because $\overline P \cap \overline {P'} \neq \emptyset$. Since the distance $\delta$ is bi-invariant and since $T$ is an isometry,
if there exist $x\in P \cap H$ and $y\in P' \cap H$, we have
\begin{align*}
&\delta_0\leq\delta(\pi_P,\pi_{P'})=\delta(\varphi_{\ell_k}(x),\varphi_{\ell_k}(y))=\delta(f^{-1}(x)f(T^{\ell_k}(x),f^{-1}(y)f(T^{\ell_k}(y))\\
&\leq \delta((f(x))^{-1}f(T^{\ell_k}(x), (f(y))^{-1} f(T^{\ell_k}(x)) +\delta((f(y))^{-1} f(T^{\ell_k}(x), (f(y))^{-1} f(T^{\ell_k}(y))\\
&< \delta_0/2+\delta_0/2,
\end{align*}
hence either $P\cap H=\emptyset$ or $P'\cap H=\emptyset$, which shows the claim.

Since ${\cal B}_{c,\ell_k}\subset {\cal C}_{c,\ell_k}$, by 3c), for each $P \in {\cal B}_{c,\ell_k}$, there exists $P'$ in ${\cal C}_{c,\ell_k}$ such that 
$\overline P\cap \overline {P'} \neq \emptyset$ and $m_{\ell_k}(P)$ and $m_{\ell_k}(P')$ have exactly one different component. By the previous claim, $P' \cap H=\emptyset$.
 
On the one hand, according to 2), the number of such distinct $P'\in {\cal C}_{c,\ell_k}$ when $P$ ranges  in ${\cal B}_{c,\ell_k}$ is $\geq \rho^{-1}N_{1,k}$. 
Hence, by 3a) $\mu(X\setminus H) \geq \rho^{-1}N_{1,k}\,  c/\Card({\cal P}_{\ell_k})$. 

On the other hand, we have $P \subset X\setminus H$ for all $P \in {\cal C}_{c,\ell_k}\setminus {\cal B}_{c,\ell_k}$, 
hence $\displaystyle \mu(X\setminus H) \geq (N_k - N_{1,k}) {c \over \Card({\cal P}_{\ell_k})}$.
Using 3b) and $\mu(F)\geq 1-\varepsilon$, we obtain
\begin{align*}
2\varepsilon&\geq \mu(X\setminus H) \geq \max((N_k - N_{1,k}) {c \over \Card({\cal P}_{\ell_k})}, \rho^{-1}N_{1,k} {c\over\Card({\cal P}_{\ell_k})}) 
\\&= {c \over \Card({\cal P}_{\ell_k})}\max(N_k-N_{1,k},{N_{1,k}\over \rho})\geq {c \over \Card({\cal P}_{\ell_k})}\, {N_k \over 1 + \rho} \geq {c \lambda \over 1 + \rho}.
\end{align*}
This leads to a contradiction by the choice of $\varepsilon$. \eop

\pgfdeclareimage[height=10.5cm]{im1}{Delta0b} 
\begin{figure} [h]
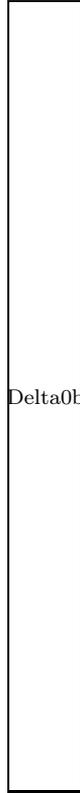
 
\pgfuseimage{im1}
\vskip -4mm
\caption{Partition ${\cal P}_\ell, \ell = 20$ ($\alpha_1= \sqrt 2, \alpha_2=e$)}
\end{figure}

\begin{thm} \label{ergoCpt} Let $(\alpha_1,\alpha_2)$ be totally irrational. If $\alpha_1$, $\alpha_2$ or $\alpha_1-\alpha_2$ is in {\rm Bad}, 
then the compact extension $\tilde T_{\alpha, \Phi_{\a}}$ is ergodic on $\T^2 \times \T$ if $\a$ is totally irrational.
\end{thm}
\proof There are three cases, which can be reduced to the case $\alpha_1\in{\rm Bad}$ as follows:
 
If $\alpha_2\in{\rm Bad}$, simply invert the roles of the first and the second component in the proof. 

\goodbreak
If $\alpha_2-\alpha_1\in{\rm Bad}$, we use that $\tilde T_{\alpha, \Phi_{\a}}$ is conjugate to 
$$\tilde T_{\beta,\Psi_{\a}}:(x_1, x_2, y)\in\T^2\times\T^d\fff (x_1+\beta_1, x_2+\beta_2, y+ 1_{\Delta_1} \, \a \, \modu),$$
where $\beta=(\beta_1,\beta_2)=(\alpha_1,\alpha_2-\alpha_1)$, $\Delta_1=\{(x_1, x_2)\in [0,1]^2: x_1+x_2\leq 1\}$
and $\Psi_{\a}= 1_{\Delta_1} \, \a$. 

Indeed, with $S:(x_1, x_2, y)\in\T^2\times \T^1\fff (x_1, x_2-x_1, y)\in\T^2\times\T^d$, if $(x_1, x_2)$ is not in the boundary of $\Delta_1$, we have
\begin{eqnarray*}
&&(S\circ T_{\alpha,s}\circ S^{-1})(x_1, x_2, y) = S(x_1+\alpha_1, x_2+x_1+\alpha_2, y + 1_{\Delta_0}(x_1,\{x_2+x_1\}) \, \a \, \modu)\\
&&=S(x_1+\alpha_1, x_2+x_1+\alpha_2, y+1_{\Delta_1}(x,y) \, \a \, \modu) \\
&&=(x_1+\alpha_1, x_2+\alpha_2-\alpha_1, y+1_{\Delta_1}(x,y) \, \a \, \modu) = \tilde T_{\beta,\Psi_{\a}}(x_1, x_2, y).
\end{eqnarray*}
Then we can prove the ergodicity of $\tilde T_{\beta,\Psi_{\a}}$ like that of $\tilde T_{\alpha,\Phi_{\a}}$.

We suppose now that $\alpha_1$ is in ${\rm Bad}$ and we use Proposition \ref{eqfunct}
with $X=\T^2$, $T=T_{\alpha}$, $G$ the group of complex numbers of modulus 1, $\cal Q=\{\Delta_0,\T^2\setminus \Delta_0\}$ 
and  $\varphi = \varphi_{\k, \a} =  \exp(2\pi i \langle \k, \a 1_{\Delta_0} \rangle)$ to conclude that the functional equation 
\begin{eqnarray}
H(Tx) = e^{2\pi i  \langle \k, 1_{\Delta_0}(x) \a, \rangle} \, H(x) \label{equa0}
\end{eqnarray}
has no measurable solution $H: X \to \T^d$ of modulus 1 for $\k \in \Z^d \setminus \{\0\}$. 

Observe that (with the notation of the proposition) $\delta_0(\varphi_{\k, \a})$ is $> 0$ since $\a$ is totally irrational.
It remains to check the hypotheses of Proposition \ref{eqfunct}.

As $\alpha_1$ is in Bad, by Lemma \ref{lem:bad} there exists a constant $c_1 >0$ such that for all $n\geq 1$ the lengths of the $n$ intervals of 
$\T^1\setminus\{0,\alpha_1,\dots,(n-1)\alpha_1\}$ are $\geq \frac{c_1}{n}$.

Let $V_0=\{0\}\times \T^1$, $H_0=\T^1\times \{0\}$ and $D_0=\{(x,x)\mod \Z^2:x\in[0,1[\}$. The boundary of  $\Delta_0$ is $V_0\cup H_0\cup D_0$.
For each $\ell>0$ consider the three sets of lines in $\T^2$,
$$\cal V_{\ell}=\{T_{\alpha}^{-k}(V_0):0\leq k<\ell\}, \, \cal H_{\ell}=\{T_{\alpha}^{-k}(H_0):0\leq k<\ell\}, \, \cal D_{\ell}=\{T_{\alpha}^{-k}(D_0):0\leq k<\ell\}.$$

Let $\cal P_\ell$ (resp. $\cal R_\ell$) be the set of connected components of $\T^2\setminus (\cup_{L\in \cal V_{\ell}\cup \cal H_{\ell}\cup\cal D_{\ell}} L)
=\T^2\setminus \cup_{0\leq i<\ell}T^{-i}(\partial \Delta_0)$ (resp. of $\T^2\setminus (\cup_{L\in \cal V_{\ell}\cup \cal H_{\ell}} L)$). (See figure 1)

We will show that $\cal P_\ell$ satisfies the assumption of the previous proposition. 
Observe first, that two points $x$ and $y$ in a same $P\in\Cal P_{\ell}$ have the same $(\Cal Q,\ell)$ coding because the translates $T_{\alpha}^k([x,y])$, 
$0\leq k<\ell$, of the segment $[x,y]$ never cross a boundary of $\Delta_0$. Next, since $\alpha$ is totally irrational, we have 
$\lim_{\ell\fff\infty} \max_{P \in {\cal P}_{\ell}} {\rm diam}(P) = 0$, hence 1) holds. 

Let $P \in \cal P_{\ell}$. It is an open convex polygon with at most $6$ edges and there are at most three lines through each vertex of $P$. 
It follows that there exist at most $6+5\times 6$ polygons $P' \in \cal P_\ell$ such that $\overline P \cap\overline {P'} \neq\emptyset$, hence 2) holds with $\rho=36$.  

It remains to find the subsequence $(\ell_k)$ and to prove that 3a), 3b) and 3c) hold.

We take $\ell_k=q_k$, where $(q_k)_{k \geq 1}$ is the sequence of denominators of $\alpha_2$.
Observe that for each $R\in\cal R_{\ell_k}$, the length of the vertical edge is $\geq \frac1{2\ell_k}$, while the length of the horizontal edge  is $\geq \frac{c_1}{\ell_k}$.
 
Let ${\cal C}_{c,\ell_k}$ be the family of polygons $\displaystyle P \in \cal P_{\ell_k}$ with a vertical edge of length $\displaystyle \geq \frac{1}{10\ell_k}$. 
 
Since the lengths of the horizontal edges of the rectangle in $\cal R_{\ell_k}$ are $\geq \frac{c_1}{\ell_k}$, the measure of any $P\in\cal P_{\ell_k}$ 
with a vertical edge of length $\displaystyle \geq \frac{1}{10\ell_k}$ is at least $\displaystyle \frac{c_2}{\ell_k^2}$ for some positive constant $c_2$ not depending on $\ell_k$.
 
One can show by induction that $\Card (\cal P_{\ell})=3\ell^2-\ell$ (actually the bounds  $c\ell^2\leq \Card (\cal P_{\ell})\leq C\ell^2$ are sufficient for the proof).  
Therefore $\displaystyle \mu(P)\geq \frac{c_2}{3\Card ( \cal P_{\ell_k})}, \forall P\in{\cal C}_{c,\ell_k}$; hence 3a).
 
For 3b), let ${\cal R}_{5,\ell_k}$ be the family of rectangles $R \in \cal R_{\ell_k}$ that contain at most $5$ elements $P\in\cal P_{\ell_k}$. 
In each of these rectangles there exists $P\in\cal P_{\ell_k}$ with a vertical edge of length $\geq \frac{1}{10\ell_k}$. 
Therefore $\Card({\cal C}_{c,\ell_k})\geq \Card({\cal R}_{5,\ell_k})$. 
Now $\Card({\cal R}_{5,\ell_k})\geq \ell_k^2/2$, because $\Card(\cal R_{\ell})=\ell^2$ and $\Card(\cal P_{\ell})\leq 3\ell^2$. It follows that 
$\displaystyle \Card ({\cal C}_{c,\ell_k})\geq \frac{1}{6} \Card(\cal P_{\ell_k})$; hence 3b).
 
At last, let $P_0 \in {\cal C}_{c,\ell_k}$. It has a vertical edge $e$ of length $\geq \frac{1}{10\ell_k}$. This edge is shared with another $P_1 \in \cal P_{\ell_k}$. 
By definition of ${\cal C}_{c,\ell_k}$, we  have $P_1 \in {\cal C}_{c,\ell_k}$. Also this edge is included in $T_{\alpha}^{-j}(\partial \Delta_0)$ for some $j\in\{0,\dots,\ell_k-1\}$. 
It follows that the arc-wise connected set $P_0 \cup e \cup P_1$ is included in $ \T^2\setminus \bigcup_{0\leq i<\ell_k,\,i\neq j}T_{\alpha}^{-i}(\partial\Delta_0)$. 
Therefore, for all $x\in P_0$, all $y\in P_1$ and all $i \neq j$ we have $T_{\alpha}^i(x)$ and $T_{\alpha}^i(y)$ both in $\Delta_0$ or both not in $\Delta_0$, 
while one exactly of the points $T_{\alpha}^j(x)$ and $T_{\alpha}^j(y)$ is in $\Delta_0$; hence 3c).
\eop

{\it Remarks}: 1) Under the assumptions of Theorem \ref{ergoCpt}, the cocycle with values in $\R^d$ generated by $(1_{\Delta_0} - \frac12) \, \a$ is not a $T_\alpha$-coboundary. 

2) Using irreducible representations, Proposition \ref{eqfunct} provides a method for the extension of Theorem \ref{ergoCpt} to skew-products by topological compact groups
instead of the torus $\T^d$.

\section{\bf Appendix}

\subsection{Badly approximable numbers and W. M. Schmidt's games} \label{Badly}

\

In this section, we explain how results of W. M. Schmidt \cite{WSch66} combined with those of J. Tseng \cite{Tseng09} or M. Einsiedler and J. Tseng \cite{EiTs11}
give an information about ``badly approximable'' numbers. 

Notice that the terminology and the notation used in this section is that of the ``Schmidt's games'': 
$\alpha$ is a number in $]0,1[$ and $\theta$ is an irrational number.

\begin{prop} \label{prop:badApprox1} Let $\theta\in\R$ be an irrational number. For $n\geq 1$, let $\mathcal B_n\subset \R^n$ be the set
$\{\beta=(\beta_1,\dots,\beta_n)\in\R^n$ such that $\beta_i \text{ and }\beta_j-\beta_i\in\operatorname{Bad_{\Z}}(\theta), \text{ for all } 1\leq i<j\leq n\}$.
Then $\dim_H\mathcal B_n=n$.
\end{prop}
\proof We use the following results of W. M. Schmidt \cite{WSch66} about $\alpha$-winning subsets in $\R$:
\hb{\it i) If $X\subset\R$ is $\alpha$-winning for some $\alpha\in]0,1[$, then $\dim_H X=1$.
\hb ii) A bi-Lipschitz image of an $\alpha$-winning subset is $\alpha$-winning.
\hb iii) Any finite or countable intersection of $\alpha$-winning subsets is $\alpha$-winning. }

We proceed by induction to prove the proposition. We have $\mathcal B_1=\operatorname{Bad_{\Z}}(\theta)$ 
which is $\tfrac18$-winning for any irrational number $\theta$, by a result of J. Tseng \cite{Tseng09}; hence $\dim_H\mathcal B_1=1$.  

Then suppose that $\mathcal B_n$ has Hausdorff dimension $n$. Let $\beta=(\beta_1,\dots,\beta_n)\in\mathcal B_n$. Consider the set 
$E_{\beta}=\operatorname{Bad_{\Z}}(\theta)\cap(\operatorname{Bad_{\Z}}(\theta)+\beta_1)\dots\cap(\operatorname{Bad_{\Z}}(\theta)+\beta_n).$

If $\beta_{n+1}\in E_{\beta}$, then $\beta_{n+1}\in \operatorname{Bad_{\Z}}(\theta)$ and for all $1\leq i\leq n$, $\beta_{n+1}-\beta_i\in\operatorname{Bad_{\Z}}(\theta)$, so that $(\beta_1,\dots,\beta_n,\beta_{n+1})\in \mathcal B_{n+1}$. 

By ii) and iii), $E_{\beta}$ is $\tfrac18$-winning, which in turn implies that $\dim_H E_{\beta}=1$.

By Corollary 7.12 in \cite{Fa97}, it follows that $\dim_H\mathcal B_{n+1}=\dim_H\mathcal B_n+1=n+1$. \eop

{\it Remark}: If $\theta \in \rm{Bad}$, since $0\in \operatorname{Bad}_{\Z}(\theta)$, the same conclusion holds when the condition $1\leq i<j\leq n$ 
above is replaced by $1\leq i, j\leq n$.

\begin{proposition} \label{prop:badApprox2} Let $\beta_1,\dots,\beta_r\in\R$.
The set of $\theta=(\theta_1,\theta_2)\in\R^2$ such that
\hb a) $1, \theta_1, \theta_2$ are linearly independent over $\Q$, 
\hb b) $\theta_1$ has bounded partial quotients (i.e., $\theta_1\in \rm{Bad}$),
\hb c) The differences $\beta_j - \beta_{j'}, \, j ,j' \in\{ 1, \dots, r\}$,  are in $\operatorname{Bad}_{\Z}(\theta_1)$,
\hb is winning and therefore has Hausdorff dimension 2.
\end{proposition}
\begin{proof} By a result of M. Einsiedler and J. Tseng (\cite[theorem 1.1]{EiTs11}), given $\beta\in\R$, the set of $\theta_1\in\R$ 
such that $\beta\in\operatorname{Bad}_{\Z}(\theta_1)$ is $\alpha$-winning, for some winning parameter $\alpha >0$ independent of $\beta$. 

By iii), it follows that the set 
$$E(\beta_1,\dots,\beta_{r})=\{\theta_1\in\R: \beta_j-\beta_{j'}\in\operatorname{Bad}_{\Z}(\theta_1), j,j'\in\{1,\dots,r_1\}\}$$  
is winning which implies  that $E(\beta_1,\dots,\beta_{r_1})\times \R$ is winning as a subset of $\R^2$. 
Since the sets of $(\theta_1,\theta_2)$ such that a) and b) hold are winning, we are done.  
\end{proof}	

\begin{cor}\label{cor:badApprox2} Let $\beta^1_1,\dots\beta^1_{r_1}, \beta^2_1,\dots\beta^2_{r_2}\in\R$. 
The set of $\theta = (\theta_1,\theta_2)\in\R^2$ such that
\hb a) $1, \theta_1, \theta_2$ are linearly independent over $\Q$, 
\hb b) $\theta_i,\theta_2$ have bounded partial quotients (i.e., $\theta_i\in \rm{Bad}$),
\hb c) the differences $\beta^i_j - \beta^i_{j'}, \, j ,j' \in\{ 1, \dots, r_i\}$, are in $\operatorname{Bad}_{\Z}(\theta_i)$, $i=1,2$,
\hb is winning and therefore has Hausdorff dimension $2$.
\end{cor}
	
\subsection{A version of the Lebesgue density theorem}

\

In this section we recall a version of the Lebesgue density theorem used in Section \ref{ergo}.

Let $(X, d)$ be a locally compact metric space equipped with a positive measure $\mu$ on the $\sigma$-algebra of its Borelian sets.

For every $n \geq 1$, let ${\cal U}_n$ be a covering (up to a set of $\mu$-measure 0) of $X$ by measurable sets of positive measure. 
We denote by $U_n(x)$ an element of ${\cal U}_n$ containing $x \in X$. Assume that the following conditions are satisfied:
 
There is a constant $C$ such that,
\begin{eqnarray}
\forall n\geq 1, \forall \, U \in \mathcal U_n, \,  \mu \bigl( \bigcup_{k\geq n}\,\bigcup_{ V \in \mathcal U_k: \, \mu(U\cap V) > 0} V \bigr) \leq C \mu(U)\label{condi1bis}
\end{eqnarray}
\begin{eqnarray} 
\lim_n {\diam}U_n(x) = 0, \forall x \in X. \label{condi2}
\end{eqnarray}
For a non negative integrable function $f$ on $(X, \mu)$, we set
$$M(f)(x) = \sup_{n \geq 1} {1\over \mu(U_n(x))} \int_{U_n(x)} f d\mu.$$
We will use the following ``Vitali covering lemma'':
\begin{lem}
Suppose that all the coverings $\mathcal U_n$ are finite and  that (\ref{condi1bis}) holds. Let $\mathcal V\subset \cup_{n\geq 1}\mathcal U_n$. 
Then there exists $\mathcal W\subset \mathcal V$ such that 
\begin{enumerate}
\item[i)] for all $U, U'\in \mathcal W$, $U\ne U'\implies \mu(U\cap U')=0$,
\item[ii)] $\mu(\cup_{U\in \mathcal V}U)\leq C\mu(\cup_{U\in \mathcal W}U)$. 
\end{enumerate}	
\end{lem}
\proof Let us define inductively a sequence of subsets $\mathcal W_n\subset \mathcal V\cap \mathcal U_n$. Let $\mathcal W_1$ be a maximal subset of $\mathcal V\cap \mathcal U_1$ such that for all $U\ne U'\in\mathcal W_1$, $\mu(U\cap U')=0$. Let $\mathcal W_{n+1}$ be a  maximal subset of $\mathcal V\cap\mathcal U_{n+1}$ such that 
for all $U\in \mathcal W_{n+1}$, $\mu(U\cap U')=0$ whenever $U'\in\mathcal W_{n+1}$ and $U'\ne U$, or $U'\in \cup_{1\leq k\leq n}\mathcal W_k$. 
Let $\mathcal W=\cup_{n\geq 1} \mathcal W_n$. 

Clearly i) holds. Next, if $V\in\mathcal V\cap \mathcal U_n$, then $V$ is either  in $\mathcal W$ or  $V$ cannot be add to $\mathcal W_n$, 
so that there exists $U\in\cup_{k\leq n}\mathcal W_k$ such that $\mu(U\cap V)>0$. It follows that any $V\in\mathcal V\cap\mathcal U_n$ is included in
$$\bigcup_{1\leq k\leq n}\bigcup_{U\in\mathcal W_k}\bigcup_{m\geq k}\,\bigcup_{ W\in\mathcal U_m:\mu(W\cap U)>0)}W.$$
Therefore, by (\ref{condi1bis}),
$$\mu(\cup_{V\in\mathcal V}V)\leq \sum_{k\geq 1}\sum_{U\in \mathcal W_k}\mu(\bigcup_{m\geq k}\,\bigcup_{ W\in\mathcal U_m:\mu(W\cap U)>0)}W)
\leq \sum_{k\geq 1}\sum_{U\in \mathcal W_k}C\mu(U). $$ 
Thanks to i), we obtain ii).
\eop
\begin{lem} \label{maxim} Under Condition (\ref{condi1bis}), for all positive $\lambda > 0$ and all non negative integrable functions $f$,
\begin{eqnarray}
\mu \{M(f) > \lambda\} \leq C{\|f\|_1 \over \lambda}. \label{inegMax}
\end{eqnarray}
\end{lem}
\proof For all $x \in A := \{M(f) > \lambda\}$, there is an integer $r(x) \geq 1$ such that
\begin{eqnarray}
\int_{U_{r(x)}(x)} f d\mu > \lambda \, \mu(U_{r(x)}(x)). \label{ineq1}
\end{eqnarray}
Let $\mathcal V=\{U_{r(x)}(x):x\in A\}$. By the previous lemma, there exists $\mathcal W \subset \mathcal V$ such that i) and ii) hold. Therefore,
$$\|f\|_1\geq \int_{\cup_{U\in\mathcal W}U}f\,d\mu\overset{\text{by i}}
{=}\sum_{U\in\mathcal W}\int_Uf\,d\mu>\lambda\sum_{U\in\mathcal W}\mu(U)\overset{\text{by ii}}
{\geq}\frac{\lambda}{C}\mu(\cup_{V\in\mathcal V}V)\geq \frac{\lambda}{C}\mu(A). \eop$$

\begin{thm} \label{densL} Under Conditions (\ref{condi1bis}) and (\ref{condi2}), for all $f$ in $L^1(\mu)$, we have
$$\lim_{n \to \infty} {1 \over \mu(U_n(x))} \int_{U_n(x)} f d\mu = f(x), \text{ for } \mu\text{-a.e. } x \in X.$$ 
In particular if $B$ is a measurable set of positive measure, for every $\varepsilon$, there are $n(\varepsilon)$ and $B_\varepsilon \subset B$ 
of measure $\geq \frac12 \mu(B)$ such that:
$$\mu(B \cap U_n(x)) \geq (1-\varepsilon) \mu(U_n(x)), \forall n \geq n(\varepsilon), \, \forall x \in B_\varepsilon.$$
\end{thm}
\proof It is enough to prove that, for all $f \in L^1(\mu)$, for almost all $x$,
$$f^*(x) := \limsup_n{1 \over \mu(U_n(x))} \int_{U_n(x)} |f(y) - f(x)| d\mu(y) =0.$$
For all $\varepsilon > 0$, there exists a continuous function $g_\varepsilon \in L^1(\mu)$ such that $\|f - g_\varepsilon \|_1 \leq \varepsilon$. 

By Condition (\ref{condi2}), for all $x$ there is an integer $N(x)$ such that, for each $n \geq N(x)$, the variation of $g_\varepsilon$ on $U_n(x)$ is less than $\varepsilon$.
Therefore, with $h_\varepsilon = f - g_\varepsilon$, we have for $n \geq N(x)$:
\begin{eqnarray*}
&&{1 \over \mu(U_n(x))} \int_{U_n(x)} |f(y) - f(x)| d\mu(y) \\
&&\leq {1 \over \mu(U_n(x))} \int_{U_n(x)} |h_\varepsilon(y) - h_\varepsilon(x)| d\mu(y) 
+{1 \over \mu(U_n(x))} \int_{U_n(x)} |g_\varepsilon(y) - g_\varepsilon(x)| d\mu(y) \\
&& \leq M(|h_\varepsilon|)(x) + |h_\varepsilon (x)| + \varepsilon.
\end{eqnarray*}
Hence, for all $x$ and $\varepsilon > 0$, $f^*(x) \leq M(|h_\varepsilon|)(x) +|h_\varepsilon (x)| + \varepsilon$.

For $\lambda > 0$, taking $\varepsilon$ such that $0 < \varepsilon < \lambda$, it follows by Lemma \ref{maxim} that
$$\mu(f^* > \lambda) \leq  \mu(M(|h_\varepsilon|) > {\lambda - \varepsilon \over 2}) + \mu(|h_\varepsilon| > {\lambda - \varepsilon \over 2})
\leq 2 (1 + C){\|h_\varepsilon\|_1 \over \lambda - \varepsilon} \leq \varepsilon {2 (1 + C) \over \lambda - \varepsilon} \to  0,$$ 
when $\varepsilon$ goes to zero. As $\lambda$ is arbitrary $> 0$, this implies $f^* = 0$ a.e. 

The last assertion follows by taking $f = 1_B$.\eop

\end{document}